\RequirePackage[l2tabu, orthodox]{nag}

\documentclass[12pt]{amsart}
\usepackage{fullpage,url,amssymb,enumerate}
\usepackage[all]{xy} 
\usepackage{comment}
\usepackage{graphicx}
\usepackage{breqn}

\usepackage[OT2,T1]{fontenc}

\usepackage{amsthm}

\usepackage{color}

\def\cC{\mathcal{C}}

\def\cE{\mathcal{E}}
\def\cF{\mathcal{F}}
\def\cG{\mathcal{G}}
\def\cH{\mathcal{H}}
\def\cI{\mathcal{I}}
\def\cL{\mathcal{L}}
\def\cM{\mathcal{M}}

\def\cO{\mathcal{O}}
\def\bP{\mathbb{P}}
\def\cP{\mathcal{P}}

\def\cV{\mathcal{V}}
\def\cW{\mathcal{W}}
\def\bZ{\mathbb{Z}}

\def\wt{\widetilde}

\DeclareMathOperator{\Bl}{Bl}

\DeclareMathOperator{\Gr}{Gr}

\DeclareMathOperator{\Num}{Num}
\DeclareMathOperator{\Pic}{Pic}
\DeclareMathOperator{\pr}{pr}
\DeclareMathOperator{\Proj}{Proj}

\DeclareMathOperator{\Spec}{Spec}
\DeclareMathOperator{\Sym}{Sym}

\newtheorem{thm}{Theorem}[section]
\newtheorem{lem}[thm]{Lemma}
\newtheorem{cor}[thm]{Corollary}
\newtheorem{prop}[thm]{Proposition}

\newtheorem{question}{Question}

\newtheorem{rem}[thm]{Remark}

\newtheorem{defn}[thm]{Definition}

\makeatletter
\g@addto@macro\bfseries{\boldmath} 
\makeatother

\usepackage{microtype}

\usepackage[
	backref,
	pdfauthor={Carl Lian}, 
	pdftitle={Enumerating pencils with moving ramification on curves},
]{hyperref}

\begin{document}

\title{Enumerating pencils with moving ramification on curves}

\author{Carl Lian}
\address{Institut f\"{u}r Mathematik, Humboldt-Universit\"{a}t zu Berlin, 12489 Berlin, Germany}
\email{liancarl@hu-berlin.de}
\urladdr{\url{https://sites.google.com/view/carllian}}

\date{\today}

\begin{abstract}
We consider the general problem of enumerating branched covers of the projective line from a fixed general curve subject to ramification conditions at possibly moving points. Our main computations are in genus 1; the theory of limit linear series allows one to reduce to this case. We first obtain a simple formula for a \textit{weighted} count of pencils on a fixed elliptic curve $E$, where base-points are allowed. We then deduce, using an inclusion-exclusion procedure, formulas for the numbers of maps $E\to\bP^1$ with moving ramification conditions. A striking consequence is the invariance of these counts under a certain involution. Our results generalize work of Harris, Logan, Osserman, and Farkas-Moschetti-Naranjo-Pirola.
\end{abstract}

\maketitle


\section{Introduction}\label{intro}

\begin{question}\label{main_question}
Let $(C,p_1,\ldots,p_n)$ be a general pointed curve of genus $g$, where $2g-2+n>0$. Let $d,d_1,d_2,\ldots,d_{n+m}$ be integers such that $2\le d_i\le d$ for all $i$. How many $(m+1)$-tuples $(p_{n+1},\ldots,p_{n+m},f)$ are there, where $p_1,p_2,\ldots,p_{n+m}\in C$ are pairwise distinct points, and $f:C\to\bP^1$ is a morphism of degree $d$ with ramification index at least $d_i$ at each $p_i$?
\end{question}

In other words, we count $f:C\to\bP^1$ (up to automorphisms of the target) subject to ramification conditions at $n$ fixed points and $m$ moving points. According to a naïve dimension count, we should expect the answer to be a positive integer when 
\begin{equation}\label{total_ramification_conditions}
g+2(d-g-1)=\sum_{i=1}^{n}(d_i-1)+\sum_{i=n+1}^{n+m}(d_i-2).
\end{equation}
Indeed, under the genericity assumption, the associated moduli problem has dimension zero if and only if (\ref{total_ramification_conditions}) holds. Comparison with the Riemann-Hurwitz formula shows that $m\le 3g$; in fact, by introducing additional moving simple ramification points, one may assume $m=3g$, at the cost of multiplying the answer to Question \ref{main_question} by  $(3g-m)!$.

Various special cases of Question \ref{main_question} have been addressed in the literature, and arise naturally in the study of cycles on moduli spaces of curves. Formulas were given in the case $m=0$ by Osserman \cite{osserman}, and the case $(n,m)=(1,1)$ by Logan \cite[Theorem 3.2]{logan}. The case $g=1$, $n=1$, $m=3$, $(d_1,d_2,d_3,d_4)=(d,d-1,3,2)$ was established by Harris \cite[Theorem 2.1(f)]{harris}. Most recently, Farkas-Moschetti-Naranjo-Pirola \cite{fmnp} introduced \textit{alternating Catalan numbers}, counting minimal degree covers $f:C\to\bP^1$ with alternating monodromy group: this is the case where $n=0$, $d=2g+1$, and $d_i=3$ for $i=1,2,\ldots,m=3g$.

In this paper, we give an essentially complete answer to Question \ref{main_question}. First, we record the well-known answer when $g=0$, in which case we must have $m=0$:

\begin{thm}[cf. \cite{osserman}]\label{genus0_count}
Let $p_1,\ldots,p_n$ be general points on $\bP^1$. Let $d,d_1,\ldots,d_n$ be integers satisfying $d_1+\cdots+d_n=2d-2+n$. Then, the number of degree $d$ morphisms $f:\bP^1\to\bP^1$ with ramification index at least $d_i$ at $p_i$ is equal to the intersection number
\begin{equation*}
\int_{\Gr(2,d+1)}\sigma_{d_1-1}\cdots\sigma_{d_n-1}.
\end{equation*}
\end{thm}

The condition of ramification of order $d_i$ at a general point $p_i$ is parametrized by a Schubert cycle of class $\sigma_{d_i-1}\in A^{*}(\Gr(2,H^0(\bP^1,\cO(d))))$. Thus, the content of Theorem \ref{genus0_count} is that the Schubert cycles associated to general points $p_i$ intersect transversely, which follows from \cite{mtv}.

Our main new results are on an elliptic curve $(E,p_1)$, where we adapt the method of \cite{harris}. The principal difficulty is the possibility that the moving points may become equal, producing high-dimensional excess loci. We circumvent this problem by imposing the ramification conditions one at a time, and in two steps: first, impose the divisorial condition of simple ramification at $p_i$. Then, subtract the ``diagonal'' excess divisors where $p_j=p_i$, where $j<i$, and express the condition of higher ramification in terms of a contact condition of the residual divisor in the universal family of pencils on $E$.

This process introduces contributions from pencils with base-points, with multiplicities equal to certain intersection numbers on a Grassmannian. We are led to a natural weighting (see Definition \ref{pencil_weight}) on the set of pencils on $E$, and obtain:

\begin{thm}\label{genus1_weighted_count}
Let $(E,p_1)$ be a general elliptic curve. Let $d,d_1,d_2,d_3,d_4$ be integers such that $d\ge2$, $1\le d_i\le 2d+1$ and $d_1+d_2+d_3+d_4=2d+4$. Then, the \textit{weighted} number of 4-tuples $(V,p_2,p_3,p_4)$, where the $p_i\in E$ are pairwise distinct points, and $V$ is a pencil on $E$ of degree $d$ with total vanishing at least $d_i$ at $p_i$, is
\begin{equation*}
\wt{N}_{d_1,d_2,d_3,d_4}=\frac{12C_{d-2}}{d}(d_1-1)(d_2-1)(d_3-1)(d_4-1),
\end{equation*}
where 
\begin{equation*}
C_n=\frac{1}{n+1}\binom{2n}{n}
\end{equation*}
denotes the $n$-th Catalan Number.
\end{thm}

In order to extract the answer to Question \ref{main_question} when $(g,n,m)=(1,1,3)$, we carry out a delicate inclusion-exclusion procedure, and obtain:

\begin{thm}\label{genus1_unweighted_count}
Let $(E,p_1)$ be a general elliptic curve. Let $d,d_1,d_2,d_3,d_4$ be integers such that $d\ge2$, $1\le d_i\le d$ and $d_1+d_2+d_3+d_4=2d+4$. Then, the number $N_{d_1,d_2,d_3,d_4}$ of 4-tuples $(p_2,p_3,p_4,f)$, where $p_i\in E$ are pairwise distinct points, and $f:E\to\bP^1$ is a morphism of degree $d$ with ramification index at least $d_i$ at each $p_i$, is equal to:
\begin{enumerate}
\item[(a)] The intersection number
\begin{equation*}
\int_{\Gr(2,d+1)}\left(\prod_{i=1}^{4}\sum_{a_i+b_i=d_i-2}\sigma_{a_i}\sigma_{b_i}\right)(8\sigma_{11}-2\sigma_1^2).
\end{equation*} 
\item[(b)] The constant term of the Laurent polynomial
\begin{equation*}
P_{d_1-1}P_{d_2-1}P_{d_3-1}P_{d_4-1},
\end{equation*}
where
\begin{equation*}
P_{r}=rq^{r}+(r-2)q^{r-2}+\cdots+(2-r)q^{2-r}+(-r)q^{-r}.
\end{equation*}
\item[(c)] An explicit piecewise polynomial function of degree 7 in $d_1,d_2,d_3,d_4$, see (\ref{N_explicit_case1}) and (\ref{N_explicit_case2}) of \S \ref{genus1_unweighted_explicit}.
\end{enumerate}
\end{thm}

When $(d_1,d_2,d_3,d_4)=(d,d,2,2)$, we recover the familiar fact that the number of covers $f:E\to\bP^1$ of degree $d$, totally ramified at the origin and one other point, are in bijection with the $d^2-1$ elements of $E[d]-\{p_1\}$. When $(d_1,d_2,d_3,d_4)=(d,d-1,3,2)$, we recover \cite[Theorem 2.1(f)]{harris}. When $(d_1,d_2,d_3,d_4)=(3,3,3,3),(5,3,3,3)$, we recover \cite[Theorem 4.1]{fmnp} and \cite[Theorem 4.8]{fmnp}, respectively.

We may then deduce the following ``duality."

\begin{thm}\label{genus1_duality}
Let $(E,p_1),d,d_1,d_2,d_3,d_4$ be as in Theorem \ref{genus1_unweighted_count}. Then, we have
\begin{equation*}
N_{d_1,d_2,d_3,d_4}=N_{d+2-d_1,d+2-d_2,d+2-d_3,d+2-d_4}.
\end{equation*}
\end{thm}

A similar duality was observed by Liu-Osserman in genus 0, see \cite[Question 5.1]{liuosserman}. We are not aware of a geometric explanation for this phenomenon in genus 0 or genus 1, nor whether it generalizes in any way to higher genus.

Finally, we consider the general case. As in \cite{logan,osserman,fmnp}, we degenerate to a comb curve in which the $p_1,\ldots,p_n$ specialize to general points on the rational spine, and obtain:

\begin{thm}\label{degeneration}
For any $g,d,d_1,\ldots,d_{n+m}$, the answer to Question \ref{main_question} is determined explicitly by the formulas given in Theorems \ref{genus0_count} and \ref{genus1_unweighted_count}, see Proposition \ref{degeneration_formula_precise}.
\end{thm}

While the idea is simple, the resulting degeneration formula is complicated, because in general, there are many ways to assign ramification sequences at the nodes of the comb. As a result, this approach has not yet yielded simple formulas answering Question \ref{main_question}, as in the case of genus 1.

We also remark that the methods in the proof of Theorem \ref{genus1_weighted_count} work in the general case: one can define a weighted count as in genus 1, and proceed in a similar way. However, combinatorial difficulties again arise from the fact that the number of moving points is linear in $g$. Thus, when $g\ge2$, obtaining answers to Question \ref{main_question} in the spirit of Theorems \ref{genus1_weighted_count} and \ref{genus1_unweighted_count} remains open.

The structure of this paper is as follows. We collect a series of preliminary facts in \S \ref{preliminary_section}. We develop the main geometric input in \S \ref{weighted_section}, proving Theorem \ref{genus1_weighted_count}. \S \ref{unweighted_section} is purely combinatorial: here we deduce Theorem \ref{genus1_unweighted_count} and Theorem \ref{genus1_duality} from Theorem \ref{genus1_weighted_count}. Finally, we explain the degeneration method in \S \ref{degeneration_section}, giving a precise version of Theorem \ref{degeneration}.

\vskip 5pt

\noindent {\bf Acknowledgments:} I am grateful to my advisor, Johan de Jong, for numerous ideas and stimulating discussions throughout the course of this project. I also thank Amol Aggarwal, Andrei Bud, Dawei Chen, Henry Liu, Melissa Liu, Brian Osserman, Nicola Tarasca, and Michael Thaddeus for helpful comments and conversations, as well as the anonymous referee for numerous corrections. This project was undertaken with the support of an NSF Graduate Research Fellowship.

\section{Preliminaries}\label{preliminary_section}
\subsection{Conventions}
We work over an algebraically closed field $k$ of characteristic zero.

If $V$ is a vector space, $\bP V$ denotes the variety $\Proj(\Sym^{*}V^\vee)$, parametrizing lines in $V$. More generally, if $\cV$ is a vector bundle over a scheme, we follow the same convention. Similarly, $\Gr(r,V)$ is the Grassmannian of $r$-planes in $V$.

Let $\cL$ be a line bundle on a smooth curve $C$ and let $(V,\cL)$ be a linear series on $C$, that is, $V\subset H^0(C,\cL)$. We will often refer abusively to the linear series $(V,\cL)$ as simply $V$. In this paper, $V$ will always have rank 1, that is, the dimension of the vector space $V$ is 2. The \textbf{vanishing sequence} of $V$ at a point $p\in C$ is the pair $(a_0,a_1)$ such that, in terms of some analytic local coordinate $x$ around $p$, the sections of $V$ are $x^{a_0},x^{a_1}$, where $a_1>a_0\ge0$ are integers. The \textbf{total vanishing} of $V$ at $p$ is the integer $a_0+a_1$, and $V$ has a \textbf{base-point} at $p$ if and only if $a_0\ge1$. If $a_0=0$, the \textbf{ramification index} of $V$ at $x$ is $a_1$; we also say that $V$ is \textbf{ramified to order $a_1$} at $p$. These same definitions make sense when $V$ is a limit linear series on a compact type curve $C$, and $p\in C$ is a smooth point.

Suppose that $V$ has degree $d$, that is, the degree of the underlying line bundle $\cL$ is equal to $d$. The \textbf{Brill-Noether number} of $V$ respect to marked points $p_i\in C$ at which $V$ has vanishing sequence $(a_i,b_i)$ for $i=1,2,\ldots,n$ is
\begin{equation*}
\rho(V,\{p_i\})=g+2(d-g-2)-\sum_{i=1}^{n}(a_i+b_i-1).
\end{equation*}
If $V$ is a limit linear series on a compact type curve $C$ on which the $p_i$ are smooth points, the same definition makes sense. Let $C_0\subset C$ be an irreducible sub-curve of $C$. We will denote the Brill-Noether number of the $C_0$-aspect of $V$ with respect to the marked points and nodes on $C_0$ by $\rho(V,\{p_i\})_{C_0}$. A straightforward computation shows that when $V$ is a \textit{crude} limit linear series (in the sense of \cite{eh_lls}), we have
\begin{equation}\label{bn_number_additive}
\rho(V,\{p_i\})\ge\sum_{C_0\subset C}\rho(V,\{p_i\})_{C_0},
\end{equation}
with equality if and only if $V$ is a \textit{refined} limit linear series.

We consider counts of morphisms $f:C\to\bP^1$ up to automorphisms of the target. Thus, it is equivalent to count isomorphism classes of base-point-free pencils (linear series of rank 1) on the fixed curve $C$.

If $F(q)$ is a power series in $q$, we denote the coefficient of $q^d$ by $F(q)[q^d]$. If $\alpha\in A^{*}(X)$ is a Chow class on some variety $X$, then $\{\alpha\}_d$ denotes its projection to $A^d(X)$.

\vskip 3pt

\subsection{Numerology}\label{numerology} Here, we collect the numerical conditions in order for Question \ref{main_question} to have interesting answers.

The celebrated Brill-Noether theorem states that the moduli space of linear series of degree $d$ and rank $r$ and on a general curve $C$ has dimension $\rho(d,g,r)=g+(r+1)(d-g-r)$, and moreover that loci determined by ramification conditions at fixed general points of $C$ have the expected codimension, see \cite[Theorem 4.5]{eh_lls}. However, ramification conditions at \textit{moving} points may fail to impose the expected number of conditions, that is, Brill-Noether loci in $\cM_{g,n}$ may have lower-than-expected codimension, see \cite[\S 2]{eh_counterexamples}.

On the other hand, owing to the existence of well-behaved Hurwitz spaces, moving ramification conditions impose the correct number of conditions in the case $r=1$. We summarize this in the following well-known proposition:

\begin{prop}\label{expected_dim}
Let $(C,p_1,\ldots,p_n)$ be a general marked curve of genus $g$, where $2g-2+n>0$. Let $d\ge2$ be an integer, and let $(a^i_0,a^i_1)$, $i=1,2,\ldots,n+m$ be ordered pairs of integers satisfying $0\le a^i_0<a^i_1\le d$ for $i=1,2,\ldots,n+m$ and $a^i_1>1$ for $i=n+1,\ldots,n+m$. Let $\cG$ be the moduli space of tuples $(V,p_{n+1},\ldots,p_{n+m})$, where the $p_i\in C$ are pairwise distinct points, and $V$ is a pencil with vanishing sequence at least $(a^i_0,a^i_1)$ at $p_i$ for $i=1,2,\ldots,n+m$. Then, $\cG$ is pure of the expected dimension
\begin{equation*}
    \rho'=g+2(d-g-1)-\sum_{i=1}^{n}(a^i_0+a^i_1-1)-\sum_{i=n+1}^{n+m}(a^i_0+a^i_1-2).
\end{equation*}
In particular, if $\rho'<0$, then $\cG$ is empty.
\end{prop}

\begin{proof}
We may assume by twisting $V$ and decreasing $d$ that the $a_0^i=0$ for all $i$. We recall the following classical fact: Hurwitz spaces of covers $C\to\bP^1$ with prescribed ramification profiles are \'{e}tale over the spaces $M_{0,r}$ parametrizing branch divisors on $\bP^1$. This is essentially the content of the Riemann Existence Theorem, but alternatively follows by passing to (the appropriate components of) the smooth stack $B_{g,n}(S_d)$ of twisted $S_d$-covers, which agrees with the Hurwitz space on the locus of curves of smooth curves, see \cite[Theorems 3.0.2, 4.3.2]{acv}

In particular, the Hurwitz spaces have the expected dimension, so the spaces $\cG$ do as well. We omit the details.
\end{proof}

Thus, in Question \ref{main_question}, we impose the condition (\ref{total_ramification_conditions}).

\begin{cor}\label{enumerated_covers_are_generic}
Suppose that (\ref{total_ramification_conditions}) holds. Then, all of the morphisms counted in Question \ref{main_question} are pairwise distinct, have ramification index exactly $d_i$ at $p_i$, and have ramification index at most 2 away from $p_1,p_2,\ldots,p_{n+m}$. 
\end{cor}

\begin{prop}\label{number_of_moving_points}
Suppose that (\ref{total_ramification_conditions}) holds. Then, $m\le 3g$.
\end{prop}

\begin{proof}
By Riemann-Hurwitz, we have
\begin{align*}
2d+2g-2&\ge\sum_{i=1}^{n+m}(d_i-1)\\
&=g+2(d-g-1)+m,
\end{align*}
where we have applied (\ref{total_ramification_conditions}) in the second line. Rearranging yields $m\le 3g$.
\end{proof}

By the last part of Corollary \ref{enumerated_covers_are_generic}, we may add additional moving points $p_i$ with $d_i=2$, where $m+n+1\le i\le m+3g$, without changing condition \ref{total_ramification_conditions}. From the proof of Proposition \ref{number_of_moving_points}, $f:C\to\bP^1$ is unramified away from the $p_i$. With these additional moving points, the answer to Question \ref{main_question} is multiplied by a factor of $(3g-m)!$, the number of ways to label the additional simple ramification points. We will therefore assume throughout the rest of the paper that $m=3g$, and that all ramification of $f$ occurs at the $p_i$.

We will also need a version of Proposition \ref{expected_dim} for pencils with restricted underlying line bundle. For simplicity, we stick to the following special case.

\begin{prop}\label{expected_dim_fixedL}
Let $E$ be a smooth curve of genus 1. Let $d\ge2$ be an integer, and let $(a^i_0,a^i_1)$, $i=1,2,\ldots,m$ be ordered pairs of integers satisfying $0\le a^i_0<a^i_1\le d$ and $a^i_1>1$ for $i=1,2,\ldots,m$. Let $\cG_{\cL}$ be the moduli space of tuples $(V,p_1,\ldots,p_{m})$, where the $p_i\in E$ are pairwise distinct points, and $V$ is a pencil with vanishing sequence at least $(a^i_0,a^i_1)$ at $p_i$ for $i=1,2,\ldots,m$, and the underlying line bundle of $V$ is isomorphic to $\cL$. Then, $\cG_{\cL}$ is pure of the expected dimension
\begin{equation*}
    \rho'=g+2(d-g-1)-\sum_{i=1}^{m}(a^i_0+a^i_1-2)-1.
\end{equation*}
In particular, if $\rho'<0$, then $\cG_{\cL}$ is empty.
\end{prop}

\begin{proof}
Fix a general point $p'_1\in E$. Let $\cG$ be the moduli space of tuples $(V,p_2,\ldots,p_m)$ with the same vanishing conditions as before at $p_2,\ldots,p_m$, and the vanishing conditions at $p_1$ imposed at $p'_1$, with no condition on the underlying line bundle of $V$. By Proposition \ref{expected_dim}, $\cG$ is pure of the expected dimension $\rho'$.

We have a map $\varphi:\cG_{\cL}\to\cG$ sending $(V,p_1,\ldots,p_m)$ to $t_{p_1}^{*}(V,p_2,\ldots,p_m)$, where $t_{p_1}$ denotes the translation by $p_1$ according to the group law on the elliptic curve $(E,p'_1)$. We have that $\varphi$ is a $E[d]$-torsor, where $E[d]$ denotes the group of $d$-torsion points of $E$: indeed, if $\cL'$ is the underlying line bundle of $V$, then
\begin{equation*}
\varphi^{-1}(V,p_2,\ldots,p_m)\cong\{p_1\in E|t_{p_1}^{*}\cL'\cong\cL\}.
\end{equation*}
In particular, $\dim(\cG_{\cL})=\dim(\cG)=\rho'$.
\end{proof}

\vskip 3pt

\subsection{Schubert Calculus}

Let $V$ be a vector space of dimension $n$, and fix a complete flag $0=V_0\subset V_1\subset\cdots\subset V_n=V$, where $\dim V_k=k$. On the Grassmannian $\Gr(2,n)$, let $\sigma_{a,b}\in A^{a+b}(\Gr(2,n))$ denote the class of the subscheme parametrizing two-dimensional subspaces $W\subset V$ satisfying $W\cap V_{n-1-a}\neq\{0\}$ and $W\subset V_{n-b}$. As is conventional, we denote $\sigma_a=\sigma_{a,0}$. The classes $\sigma_{a,b}$, where $0\le b\le a\le n-2$, form a $\bZ$-basis for the Chow ring $A^{*}(\Gr(2,n))$.

The following is a consequence of the Pieri Rule and Hook Length Formula:

\begin{lem}\label{LR_coeffs}
We have
\begin{equation*}
\sigma_1^{k}=\sum_{a+b=k}c_{a,b}\sigma_{a,b},
\end{equation*}
where
\begin{equation*}
c_{a,b}=\binom{a+b}{a}\cdot\frac{a-b+1}{a+1}
\end{equation*}
is the number of Standard Young Tableaux (SYT) of shape $(a,b)$.
\end{lem}

We also have the following generating function formula for the $c_{a,b}$:

\begin{lem}\label{catalan_gf}
For $t\ge1$, we have
\begin{equation*}
f_{t}(z)=\sum_{m_i=0}^{\infty}c_{t+m_{i}-1,m_{i}}z^{m_{i}}=\left(\frac{1-\sqrt{1-4z}}{2z}\right)^t.
\end{equation*}
\end{lem}

\begin{proof}
We proceed by induction on $t$. When $t=1$, we have that $c_{m_{i},m_{i}}$ is the Catalan number $C_{m_{i}}$, and 
\begin{equation*}
\sum_{m_i=0}^{\infty}c_{m_{i},m_{i}}z^{m_{i}}=\frac{1-\sqrt{1-4z}}{2z},
\end{equation*}
see, for example, \cite[Example 6.2.6]{stanley}. When $t=2$, we have that $c_{m_{i}+1,m_{i}}$ is the Catalan number $C_{m_{i}+1}$, so
\begin{align*}
\sum_{m_i=0}^{\infty}c_{m_{i}+1,m_{i}}z^{m_{i}}&=\frac{1}{z}\left(\frac{1-\sqrt{1-4z}}{2z}-1\right)\\
&=\left(\frac{1-\sqrt{1-4z}}{2z}\right)^2.
\end{align*}
When $t\ge3$, we have
\begin{equation*}
c_{t+m_{i}-1,m_{i}}=c_{t+m_{i}-1,m_{i}+1}-c_{t+m_{i}-2,m_{i}+1},
\end{equation*}
as a SYT of shape $(t+m_{i}-1,m_{i}+1)$ has its largest entry in the right-most box of either the top or bottom row. Therefore,
\begin{equation*}
f_{t}(z)=\frac{f_{t-1}(z)-1}{z}-\frac{f_{t-2}(z)-1}{z},
\end{equation*}
as $c_{n,0}=1$ for all $n$. The lemma now follows from the fact that $\alpha=\frac{1-\sqrt{1-4z}}{2z}$ satisfies the quadratic equation $z\alpha^2-\alpha+1=0$.
\end{proof}

\vskip 4pt

\section{The weighted count in genus 1: Proof of Theorem \ref{genus1_weighted_count}}\label{weighted_section}

In this section, we consider Question \ref{main_question} in the case $g=n=1$, so that $m=3$: we refer to the fixed curve as $(E,p_1)$ to emphasize that its genus is 1. Fix integers $d,d_1,d_2,d_3,d_4$ such that $2\le d_i\le 2d-2$ and $d_1+d_2+d_3+d_4=2d+4$ (we comment on the additional boundary cases allowed in Theorem \ref{genus1_weighted_count} at the end of this section). 

\subsection{The weighted count $\wt{N}_{d_1,d_2,d_3,d_4}$}

Let us now define the weighted count appearing in Theorem \ref{genus1_weighted_count}.

\begin{defn}\label{pencil_weight}
We define $\wt{N}_{d_1,d_2,d_3,d_4}$ to be the number of 4-tuples $(p_2,p_3,p_4,V)$, where $p_i\in E$ are pairwise distinct points, and $V$ is a pencil with total vanishing at least (and thus, by Corollary \ref{enumerated_covers_are_generic}, exactly) $d_i$ at $p_i$ for $i=1,2,3,4$, such that if $V$ is a pencil with vanishing sequence $(k_i,d_i-k_i)$ at $p_i$, then $V$ is counted with multiplicity
\begin{equation*}
C^{d_1,d_2,d_3,d_4}_{k_1,k_2,k_3,k_4}=\prod_{i=1}^{4}c_{d_i-k_i-1,k_i}.
\end{equation*}
Here, the $c_{a,b}$ are as in Lemma \ref{LR_coeffs}.
\end{defn}

\begin{rem}
We digress here to illustrate the role of the weights in Definition \ref{pencil_weight} in genus 0. Consider the \textit{weighted} number of pencils with \textit{total vanishing} $d_i$ at general points $p_1,\ldots,p_n$ on $\bP^1$, with weights defined analogously as in Definition \ref{pencil_weight}. By Theorem \ref{genus0_count}, this is
\begin{align*}
&\sum_{0\le k_i<d_i/2}\int_{\Gr(2,d+1)}\prod_{i=1}^{n}c_{d_i-k_i-1,k_i}\sigma_{d_i-k_i-1,k_i}\\
=&\int_{\Gr(2,d+1)}\sum_{0\le k_i<d_i/2}\prod_{i=1}^{n}c_{d_i-k_i-1,k_i}\sigma_{d_i-k_i-1,k_i}\\
=&\int_{\Gr(2,d+1)}\prod_{i=1}^{n}\sum_{0\le k_i<d_i/2}c_{d_i-k_i-1,k_i}\sigma_{d_i-k_i-1,k_i}\\
=&\int_{\Gr(2,d+1)}\prod_{i=1}^{n}\sigma_1^{d_i-1}\\
=&\int_{\Gr(2,d+1)}\sigma_1^{2d-2}\\
=&C_{d-1},
\end{align*}
Thus, the weighted count of pencils produces a considerably simpler answer than the unweighted count of base-point free pencils; we will find a similar phenomenon in genus 1. More generally, in the weighted setting, vanishing conditions at multiple fixed points may be combined in to a vanishing condition at a single point, see Proposition \ref{consolidate_fixed_points}.
\end{rem}

\subsection{Outline of Proof}

We briefly summarize the method to compute $\wt{N}_{d_1,d_2,d_3,d_4}$. We first change the problem slightly: fix a line bundle $\cL$ on $E$. Up to a factor of $d^2$, it suffices to enumerate pencils on $E$ with underlying line bundle $\cL$ and the same ramification conditions, but where $p_1$ is allowed to move (Proposition \ref{fixL_instead}).

We then work on the parameter space $T=\Gr(2,H^0(\cL))\times E_1\times E_2\times E_3\times E_4$, where the $E_i$ are all isomorphic to $E$. We would like to consider the locus of 5-tuples $(V,p_1,p_2,p_3,p_4)$ where $V$ is ramified to order $d_i$. The main difficulty is to remove the excess loci where the $p_i$ become equal to each other: we do this as follows. First, let $T_1$ be the (closure of the) codimension 1 locus where $V$ is simply ramified at $p_1$; its class is expressed using Porteous's formula. We show in Lemma \ref{contact_locus} that the locus $T_{d_1-1}$ where $T_1$ has contact order at least $d_1-1$ with $E_1$ is, set-theoretically, the locus where $V$ has \textit{total vanishing} at least $d_1$ at $p_1$. Moreover, we show in Lemma \ref{principal_parts} that the components of $T_{d_1-1}$ parametrizing pencils with vanishing sequence at least $(k_1,d_1-k_1)$ appear with multiplicity $c_{d_1-k_1-1,k_1}$, as defined in Lemma \ref{LR_coeffs}.

Next, on $T_{d_1-1}$, we impose the condition of simple ramification at $p_2$, which defines a Cartier divisor $\wt{T}_{d_1-1,1}\subset T_{d_1-1}$. We find in Lemma \ref{D_2_formula} that $\wt{T}_{d_1,1}$ contains the diagonal locus $\Delta_{12}$, where $p_1=p_2$, with multiplicity $d_1-1$. The residual divisor $T_{d_1-1,1}=\wt{T}_{d_1-1,1}-(d_1-1)\Delta_{12}$ is the closure of the locus of $(V,p_1,p_2,p_3,p_4)$ where $p_1\neq p_2$, and $V$ has total vanishing at least $d_1$ at $p_1$ and at least 2 at $p_2$.

As in the construction of $T_{d_1-1}$, we now let $T_{d_1-1,d_2-1}$ be the locus on $T_{d_1-1}$ where $T_{d_1-1,1}$ intersects $E_2$ with multiplicity at least $d_2-1$. Set-theoretically, $T_{d_1-1,d_2-1}$ is the locus where $V$ has total vanishing at least $d_i$ at $p_i$ for $i=1,2$. We then repeat this procedure at $p_3,p_4$.

In the end, we obtain the zero-dimensional subscheme $T_{d_1-1,d_2-1,d_3-1,d_4-1}\subset\Gr(2,H^0(\cL))\times E_1\times E_2\times E_3\times E_4$ where $V$ has total vanishing at least $d_i$ at $p_i$, for $i=1,2,3,4$. The theory of limit linear series guarantees that $T_{d_1-1,d_2-1,d_3-1,d_4-1}$ is disjoint from all diagonals, as we see in Lemmas \ref{ram_points_collide} and \ref{no_excess_pencils}. The multiplicity of a component of $T_{d_1-1,d_2-1,d_3-1,d_4-1}$ is exactly its weight, as defined in Definition \ref{pencil_weight}, and integrating the class of $T_{d_1-1,d_2-1,d_3-1,d_4-1}$ over $T$ yields Theorem \ref{genus1_weighted_count}.

\vskip 3pt

\subsection{Pencils with fixed underlying line bundle}

For the rest of this section, we fix a line bundle $\cL$ of degree $d$ on $E$. Let $G=\Gr(2,H^0(E,\cL))\cong\Gr(2,d)$.

\begin{prop}\label{fixL_instead}
$\wt{N}_{d_1,d_2,d_3,d_4}$ is equal to the product of $1/d^2$ and the weighted number of 5-tuples $(V,p'_1,p_2,p_3,p_4)$, where $p'_1,p_2,p_3,p_4\in E$ are pairwise distinct, and $V$ is a pencil on $E$ with underlying line bundle $\cL$ and with total vanishing $d_1$ at $p'_1$ and $d_i$ at $p_i$ for $i=2,3,4$. Here, the weighting in the latter count is the same as in the definition of $\wt{N}_{d_1,d_2,d_3,d_4}$.
\end{prop}

\begin{proof}
This is immediate from the proof of Proposition \ref{expected_dim_fixedL}, as the fibers of $\varphi:\cG_{\cL}\to\cG$ have size $d^2$ when the (expected) dimension of the source and target are both equal to zero.
\end{proof}

In light of Proposition \ref{fixL_instead}, we will drop the fixed point $p_1$ from $E$, and by abuse of notation, count 5-tuples $(V,p_1,p_2,p_3,p_4)$ where $p_1$ is allowed to move, but the underlying line bundle of $V$ is constrained to be isomorphic to $\cL$, that is, $V\subset H^0(\cL)$.

\vskip 3pt

\subsection{The ramification loci on $G\times E$}\label{ram_loci_section}

\begin{defn}
For non-negative integers $0\le b\le a\le d-1$, let $\Sigma_{a,b}\subset G\times E$ be the closed subscheme parametrizing pairs $(V,p)$ where $V\subset H^0(\cL)$ is a pencil and $p\in E$ is a point at which the vanishing sequence of $V$ is at least $(b,a+1)$. When $a,b$ fail to satisfy $0\le b\le a\le d-1$, we declare $\Sigma_{a,b}$ to be empty, and when $b=0$, we denote $\Sigma_a=\Sigma_{a,b}$.
\end{defn}

We construct $\Sigma_{a,b}$ as follows. Let
\begin{equation*}
    \cF_k=p_{2*}(p_1^{*}\cL\otimes\cO_{E\times E}/\cI_{\Delta}^{k}),
\end{equation*}
where $p_i:E\times E\to E$ are the projection maps and $\cI_{\Delta}$ is the ideal sheaf of the diagonal $\Delta\subset E\times E$. Note that $\cF_k$ is locally free of rank $k$. Let $\cV$ be the universal 2-plane bundle over $G=\Gr(2,H^0(\cL))$ and let $\pr_G,\pr_E$ be the projection maps from $G\times E$ to $G,E$, respectively. We have natural maps
\begin{equation*}
    \varphi_k:\pr_G^{*}\cV\to\pr_E^{*}\cF_k
\end{equation*}
evaluating sections of $\cL$ to order $k$. Then, $\Sigma_{a,b}$ is the scheme-theoretic intersection $M_0(\varphi_b)\cap M_1(\varphi_{a+1})$, where $M_i(\varphi_k)$ is the degeneracy locus where $\varphi_k$ has rank at most $i$.

\begin{lem}\label{Sigma_ab_properties}
\quad
\begin{enumerate}
\item[(a)] Suppose that $0\le b\le a\le d-2$. Then, $\Sigma_{a,b}$ is integral of the expected codimension $a+b$.
\item[(b)] Suppose that $0\le b\le a=d-1$. Then, $\Sigma_{a,b}$, as a set, is the disjoint union of Schubert cycles $\sigma_{d-2,b}$ on the fibers $G\times\{q\}$, for all $q$ such that $\cL\cong\cO_E(dq)$. In particular, $\Sigma_{a,b}$ again has the expected codimension $a+b$.
\end{enumerate}
\end{lem}

\begin{proof}
In both cases, Proposition \ref{expected_dim_fixedL} implies that $\Sigma_{a,b}$ has the expected codimension.

When $a<d-1$, the restriction to the fiber of $\pr_E:G\times E\to E$ over any $q\in E$ is the usual Schubert cycle $\sigma_{a,b}$ with respect to the flag consisting of the subspaces $H^0(E,\cL(-rq))\subset V$, $r=0,1,\ldots,d-1$, which is integral of the expected codimension. Therefore, $\Sigma_{a,b}$ has the same properties. 

When $a=d-1$, a section $s\in H^0(\cL)$ can only vanish at $q$ to order $d$ if $\cL\cong\cO_{E}(dq)$, in which case the condition of vanishing to order $d-1$ is equivalent to that of vanishing to order $d$. Part (b) follows.
\end{proof}

\begin{lem}\label{sigma1_int_multiplicity}
Fix $(V,q)\in G\times E$, and suppose that $V$ has vanishing sequence $(a_0,a_1)$ at $q$. Then, the multiplicity of the intersection of $\Sigma_{1}$ with $E_{V}=\pr_G^{-1}(V)$ is $a_0+a_1-1$. 
\end{lem}

\begin{proof}
Let $x$ be an analytic local coordinate on $E_V$ near $q=V(x)$, so that $\cV|_{E_V}$ is freely generated by the sections $x^{a_0},x^{a_1}$. After restriction to $E_V$, we have that $\Sigma_1$ is the vanishing locus of
\begin{equation*}
\det
\begin{bmatrix}
x^{a_0} & x^{a_1}\\
a_0x^{a_0-1} & a_1x^{a_1-1}
\end{bmatrix}
=(a_1-a_0)x^{a_0+a_1-1},
\end{equation*}
which vanishes to order exactly $a_0+a_1-1$ at $q$ because $a_0\neq a_1$.
\end{proof}

\begin{defn}\label{contact_locus}
For integers $r\ge1$, Let $T_{r}$ be the (scheme-theoretic) locus of points $(V,q)\in G\times E$ where $\Sigma_1$ intersects $E_V$ with multiplicity at least $r$.
\end{defn}

We construct $T_r$ as follows. Let $\cW_r$ be the vector bundle of rank $r$ on $G\times E$ whose fiber over $(P,q)$ is 
\begin{equation*}
    H^0(E_P,\cO(\Sigma_1)|_{E_V})/\mathfrak{m}_{(V,q)}^{r}H^0(E_P,\cO(\Sigma_1)|_{E_V}).
\end{equation*}
Globally, 
\begin{equation*}
    \cW_r=p_{2*}(p_{1}^{*}\cO(\Sigma_1)\otimes\cO_{G\times E\times E}/\cI_{\Delta}^{r}),
\end{equation*}
where $p_{i}:G\times E\times E\to G\times E$ are the two projection maps and $\cI_{\Delta}$ is the ideal sheaf of the pullback of the diagonal under $G\times E\times E\to E\times E$. Then, the effective divisor $\Sigma_1$ defines a tautological section of $\cW_r$, and we define $T_r\subset G\times E$ to be the vanishing locus of this section. In particular, $T_1=\Sigma_1$.

As a set, Lemma \ref{sigma1_int_multiplicity} implies that $T_r$ is the locus where $V$ has total vanishing at least $r+1$. Thus, it is the union of the subschemes $\Sigma_{a,b}$ with $a+b=r$, and in particular has the expected codimension $r$. The following proposition identifies the scheme-theoretic multiplicities with which the $\Sigma_{a,b}$ appear in $T_r$.

\begin{lem}\label{principal_parts}
We have
\begin{equation}\label{sigma1_power_formula}
[\Sigma_1]^r=[T_r]=\sum_{a+b=r}c_{a,b}[\Sigma_{a,b}]
\end{equation}
in $A^{r}(G\times E)$, where the $c_{a,b}$ are as in Lemma \ref{LR_coeffs}.
\end{lem}

\begin{proof}
Because $\cI_\Delta/\cI_{\Delta}^2\cong N^{\vee}_{\Delta/G\times E\times E}$ is trivial, we may filter $\cO_{G\times E\times E}/\cI^{r}_{\Delta}$ by $r$ trivial line bundle quotients on $G\times E\times E$. As $T_r$ has expected codimension, we get 
\begin{equation*}
[T_r]=c_r(\cW_r)=\{(1+[\Sigma_1])^{r}\}_r=[\Sigma_1]^r.
\end{equation*}
establishing the first equality.

By the set-theoretic description of $T_r$, we have
\begin{equation}\label{proto_sigma1_formula}
[T_r]=\sum_{a+b=r}c'_{a,b}[\Sigma_{a,b}]
\end{equation}
for some integers $c'_{a,b}>0$. We wish to show that $c'_{a,b}=c_{a,b}$ for all $a,b$. First, consider the case in which $r<d-1$. We restrict (\ref{proto_sigma1_formula}) to the fibers $G_q$ over points $q\in E$. As we have already seen, $\Sigma_1$ and the $\Sigma_{a,b}$ restrict to the usual Schubert cycles $\sigma_1$ and $\sigma_{a,b}$ with respect to the flag of sections of $\cL$ vanishing to varying orders at $q$, so $[T_r]$ restricts to $\sigma_1^r\in A^*(G_q)$. On the other hand, in $A^{r}(G_q)$, we have the formula
\begin{equation*}
    \sigma_1^r=\sum_{a+b=r}c_{a,b}\sigma_{a,b},
\end{equation*}
by definition. Because the $\sigma_{a,b}$ are linearly independent in $A^r(G_q)$, we conclude $c_{a,b}=c'_{a,b}$ for all $a,b$.

In the case $r\ge d-1$, the above argument fails because $\Sigma_{d-1,r-d+1}$ vanishes under pullback to $G_q$. We instead argue as follows. Fix a non-trivial translation $\tau$ on $E$. Let $\cG$ be the $\Gr(2,d+1)$-bundle over $E$ whose fiber over $q$ is $H^0(E,\cL((r-d+2)\tau(q)))$. On this bundle, we may define the cycles $\Sigma_{a,b}$ in terms of vanishing conditions at $q$ in exactly the same way as before. We then have a closed embedding $\iota:G\times E\to \cG$ over $E$, sending a pencil $(P,q)$ to the pencil $(P(\tau(q)),q)$ -- that is, $\iota$ adds a base point of order $r-d+2$ at $\tau(q)\neq q$ to $P$, increasing the degree of the underlying line bundle by the same amount. 

We then obtain the formula (\ref{sigma1_power_formula}) on $\cG$ in the same way we did above, as we now have $r<\deg(\cL(r-d+2)\tau(q))-1$. The cycles $\Sigma_{a,b}$ are stable under pullback by $\iota$, so we then obtain the same formula (\ref{sigma1_power_formula}) on $G\times E$, as desired.
\end{proof}

\vskip 3pt

\subsection{Imposing ramification at additional points}\label{impose_additional_ram}
We now impose the vanishing conditions at the points $p_2,p_3,p_4$ one at a time. We work on the subscheme $T_{d_1-1}\times E_2\subset G\times E_1\times E_2$, where the subscripts denote different copies of $E$ parametrizing the $p_i$.

\begin{defn}
Let $\wt{T}_{d_1-1,1}\subset T_{d_1-1}\times E_2$ denote the subscheme parametrizing $(V,p_1,p_2)$ where $(V,p_1)\subset T_{d_1-1}$, and additionally $V$ has total vanishing at least $2$ at $p_2$. 
\end{defn}

More precisely, we construct $\wt{T}_{d_1-1,1}$ by repeating the construction of $T_1\subset G\times E_1$ on $G\times E_2$, and pulling back to $T_{d_1-1}\times E_2$. As a set, $\wt{T}_{d_1-1,1}$ includes the diagonal $\Delta_{12}$, that is, the locus where $p_1=p_2$, which has codimension 1 on every component of $T_{d_1-1}\times E_2$. Off of the diagonal, $\wt{T}_{d_1-1,1}$ parametrizes $(V,p_1,p_2)$ where $V$ has total vanishing at least $d_1$ at $p_1$ and at least 2 at $p_2$. It thus follows from Proposition \ref{expected_dim_fixedL} that $\wt{T}_{d_1-1,1}$ is a Cartier divisor on $T_{d_1-1}\times E_2$.

\begin{lem}\label{D_2_formula}
As Cartier Divisors on $T_{d_1-1}\times E_2$, we have
\begin{equation*}
    \wt{T}_{d_1-1,1}=(d_1-1)\Delta_{12}+T_{d_1-1,1},
\end{equation*}
where $\Delta_{12}$ is the pullback to $T_{d_1-1}\times E_2$ of the diagonal in $E_1\times E_2$, and $T_{d_1-1,1}$ is the scheme-theoretic closure of the locus on $T_{d_1-1}\times E_2$ where $V$ has total vanishing at least 2 at $p_2$.
\end{lem}

\begin{proof}
It suffices to show that the multiplicity of $\Delta_{12}$ in $\wt{T}_{d_1-1,1}$ is $d_1-1$. In an analytic local neighborhood of a point of $G\times E_1\times E_2$, let $f(g,e_1)$ be the equation cut out by $T_1$ on $G\times E_1\times E_2$, where $g$ is a vector of local coordinates on $G$ and $e_1$ is a local coordinate on $E_1$. Then the equations cutting out $T_{d_1-1}$ are $\frac{\partial^i}{\partial (e_1)^i}f(g,e_1)=0$, for $i=0,1,\ldots,d_1-2$.

Now, the additional equation $f(g,e_2)=0$ cuts out $\wt{T}_{d_1-1,1}$ on $T_{d_1-1}\times E_2\subset G\times E_1\times E_2$, where $e_2$ is a coordinate on $E_2$. Taylor expanding in an analytic local neighborhood of a point in $\Delta_{12}$, we have
\begin{align*}
    f(g,e_2)&=f(g,e_1-(e_1-e_2))\\
    &=\sum_{i=0}^{\infty}\left(\frac{\partial^i}{\partial (e_1)^i}f(g,e_1)\right)(e_1-e_2)^i\\
    &=(e_{1}-e_2)^{d_1-1}\sum_{i=d_1-1}^{\infty}\left(\frac{\partial^i}{\partial (e_1)^i}f(g,e_1)\right)(e_1-e_2)^{i-(d_1-1)},
\end{align*}
because on $\wt{T}^1_{d_1-1}$, we have $\frac{\partial^i}{\partial (e_1)^i}f(g,e_1)=0$ for $i=0,1,\ldots,d_1-2$. Because $e_1-e_2$ is exactly the equation cutting out $\Delta_{12}$, it is left to check that $\frac{\partial^{d_1-1}}{\partial (e_1)^{d_1-1}}f(g,e_1)$ is not identically zero on $\wt{T}^1_{d_1-1}$. This follows from Proposition \ref{expected_dim_fixedL}, as the locus of triples $(V,p_1,p_2)$ with total vanishing $d_1+1$ at $p_1$ is pure of dimension strictly less than that of $T_{d_1-1}$.
\end{proof}

\begin{lem}\label{ram_points_collide}
\quad
\begin{enumerate}
    \item[(a)] As a set, $T_{d_1-1,1}\cap\Delta_{12}\subset G\times E_1\times E_2$ is equal to the locus of triples $(V,p_1,p_2)$ where $V$ has total vanishing at least $d_1+1$ at $p=p_1=p_2$.
    \item[(b)] Suppose that $V\in G$ has vanishing sequence $(a_0,a_1)$ at $p$. Then, the multiplicity of the intersection of $T_{d_1-1,1}$ with $\{V\}\times\{p\}\times E_2$ at $(V,p,p)$ is equal to $a_0+a_1-d_1$.
\end{enumerate}
\end{lem}

\begin{proof}
We first prove (a). As a set,
\begin{equation*}
    T_{d_1-1,1}=\bigcup_j\overline{(\pr_1^{*}\Sigma_{d_1-1-j,j}\cap\pr_2^{*}\Sigma_1)-\Delta_{12}},
\end{equation*}
where $\pr_i:G\times E_1\times E_2\to G\times E_i$ are the projection maps, and the closure is taken in $G\times E_1\times E_2$ (equivalently, in $T_{d_1-1}\times E_2$). Let $S_j=\overline{(\pr_1^{*}\Sigma_{d_1-1-j,j}\cap\pr_2^{*}\Sigma_1)-\Delta_{12}}$. Identifying $\Delta_{12}\subset G\times E_1\times E_2$ with $G\times E$, we claim that the set-theoretic restriction of $S_i$ to $\Delta_{12}$ is $\Sigma_{d_1-j,j}\cup\Sigma_{d_1-1-j,j+1}$.

It suffices to check the claim pointwise, after further restriction to $G\times\{q\}\times\{q\}$, for a fixed $q\in E$. Consider the one-parameter family $p_1:X=\Bl_{q\times q}E\times E\to E$, with sections $\sigma_1,\sigma_2$ equal to the proper transforms of $\{q\}\times E$ and $\Delta$, respectively. If $p_2:X\to E$ is the second projection, the line bundle $p_2^{*}\cL$ restricts to $\cL$ on the general fiber of $p_1$, and, over $q$, to $\cL$ on the elliptic component and to $\cO_{\bP^1}$ on the rational component. 

Let $\cG_{\cL,j}$ be the moduli space of limit linear series on the fibers of $p_1$ with underlying line bundle $p_2^{*}\cL$ and with vanishing sequence at least $(j,d_1-j)$ along $\sigma_1$ and at least $(0,2)$ along $\sigma_2$. Following \cite{eh_lls}, $\cG_{\cL,j}$ may be constructed as a closed subscheme of a product of Grassmannian bundles over $E$, and carries a projection map $\pi:\cG_{\cL,j}\to G\times E$ remembering the aspects of limit linear series on the elliptic components. In particular, $\pi$ is proper, so the image of $\pi$, when restricted to $G\times\{q\}$, contains $S_i$.

The fiber of $\cG_{\cL,j}$ over $q$ is the space of limit linear series $V$ on $E\cup\bP^1$, where the $E$-aspect of $V$ has underlying line bundle $\cL$, and $V$ has vanishing at least $(j,d_1-j)$ and $(0,2)$ at $p_1,p_2\in\bP^1$, respectively. A straightforward calculation shows that the $E$-aspect of $V$ has vanishing at least $(i,d_1-i+1)$ or $(j+1,d_1-j)$ at $q$. Thus, we conclude that $S_j\subset\Sigma_{d_1-j,j}\cup\Sigma_{d_1-1-j,j+1}$.

In fact, this inclusion must be an equality, because the cycle class of $S_j$ when restricted to general fiber of $G\times E\to E$ is
\begin{equation*}
    \sigma_1\sigma_{d_1-1-j,j}=\sigma_{d_1-j,j}+\sigma_{d_1-1-j,j+1},
\end{equation*}
and thus the same is true over $q$. Taking the union over all $j$ yields (a).

The statement in part (b) follows from Lemmas \ref{sigma1_int_multiplicity} and \ref{D_2_formula}. Namely, the same proof from Lemma \ref{sigma1_int_multiplicity} shows that $T_{d_1-1}\times E_2$ intersects $\{P\}\times\{q_1\}\times E_2$ at $(P,q_1,q_2)$ with multiplicity $a_0+a_1-1$. By Lemma $\ref{D_2_formula}$, the contribution from $T_{d_1-1,1}$ is $a_0+a_1-1-(d_1-1)=a_0+a_1-d_1$.
\end{proof}

We now proceed as in Definition \ref{contact_locus} and Lemma \ref{principal_parts}. Let $T_{d_1-1,d_2-1}$ be the locus on $T_{d_1-1}\times E_2$ where the divisor $T_{d_1-1,1}$ intersects the fibers of the projection $T_{d_1-1}\times E_2\to T_{d_1-1}$ with multiplicity at least $d_2-1$. 

On $\Delta_{12}$, by Lemma \ref{ram_points_collide}, the underlying set of $T_{d_1-1,d_2-1}$ is the locus of pencils with total vanishing at least $d_1+d_2-1$. Away from $\Delta_{12}$, the underlying set of $T_{d_1-1,d_2-1}$ is the locus of pencils with total vanishing at least $d_i$ at $p_i$ for $i=1,2$. It follows from Proposition \ref{expected_dim_fixedL}, that $T_{d_1-1,d_2-1}$ has the expected dimension. We therefore obtain the following analogue of Lemma \ref{principal_parts}:

\begin{lem}\label{principal_parts_twopoints}
We have
\begin{equation*}
[T_{d_1-1,1}]^{d_2-1}=[T_{d_1-1,d_2-1}]=\sum_{a+b=d_2-1}c_{a,b}\pr_2^{*}[\Sigma_{a,b}]
\end{equation*}
in $A^r(T_{d_1-1}\times E_2)$.
\end{lem}

By the push-pull formula, we conclude:

\begin{cor}\label{two_conditions_locus}
We have
\begin{equation*}
[T_{d_1-1,d_2-1}]=\pr_{1}^{*}[\Sigma_1]^{d_1-1}\cdot(\pr_{2}^{*}[\Sigma_1]-(d_1-1)[\Delta_{12}])^{d_2-1}
\end{equation*}
in $A^r(G\times E_1\times E_2)$.
\end{cor}

We may repeat this procedure with the additional conditions at $p_3,p_4$ to obtain a subscheme
\begin{equation*}
T_{d_1-1,d_2-1,d_3-1,d_4-1}\subset G\times E_1\times E_2\times E_3\times E_4
\end{equation*}
of class
\begin{dmath}\label{class_to_integrate}
\pr_{1}^{*}[\Sigma_1]^{d_1-1}\cdot(\pr_{2}^{*}[\Sigma_1]-(d_1-1)[\Delta_{12}])^{d_2-1}\cdot(\pr_{3}^{*}[\Sigma_1]-(d_1-1)[\Delta_{13}]-(d_2-1)[\Delta_{23}])^{d_3-1}\cdot(\pr_4^{*}[\Sigma_1]-(d_1-1)[\Delta_{14}]-(d_2-1)[\Delta_{24}]-(d_3-1)[\Delta_{34}])^{d_4-1}
\end{dmath}
in $A^{*}(G\times E_1\times E_2\times E_3\times E_4)$, where $\pr_i:G\times E_1\times E_2\times E_3\times E_4\to G\times E_i$ denotes the projection as before. By construction, on the locus where the $p_i\in E$ are pairwise disjoint, $T_{d_1-1,d_2-1,d_3-1,d_4-1}$ is the subscheme parametrizing $(V,p_1,p_2,p_3,p_4)$ such that $V$ has total vanishing at least $d_i$ at $p_i$. The following two lemmas show that in fact $T_{d_1-1,d_2-1,d_3-1,d_4-1}$ has the desired structure to obtain the weighted counts $\wt{N}_{d_1,d_2,d_3,d_4}$.

\begin{lem}\label{no_excess_pencils}
The subscheme $T_{d_1-1,d_2-1,d_3-1,d_4-1}$ has dimension 0, and is disjoint from all diagonals of $G\times E_1\times E_2\times E_3\times E_4$.
\end{lem}

\begin{proof}
The expected dimension of $T_{d_1-1,d_2-1,d_3-1,d_4-1}$ is 0, and the natural extensions of Lemma \ref{ram_points_collide}, along with Proposition \ref{expected_dim_fixedL} imply the first statement. Along the diagonals, a limit linear series argument similar to that in the proof of Lemma \ref{ram_points_collide} shows that the underlying set of $T_{d_1-1,d_2-1,d_3-1,d_4-1}$ consists of pencils on $E$ whose ramification is concentrated at three or fewer points. Proposition \ref{expected_dim_fixedL} implies that for a general $E$, there are no such pencils, so the lemma follows.
\end{proof}

\begin{lem}\label{identify_pencil_weights}
We have
\begin{equation*}
\wt{N}_{d_1,d_2,d_3,d_4}=\frac{1}{d^2}\int_{G\times E_1\times E_2\times E_3\times E_4}[T_{d_1-1,d_2-1,d_3-1,d_4-1}]
\end{equation*}
\end{lem}

\begin{proof}
By Proposition \ref{fixL_instead}, it suffices to show that each $(V,p_1,p_2,p_3,p_4)\in T_{d_1-1,d_2-1,d_3-1,d_4-1}$ appears with multiplicity equal to that given in Definition \ref{pencil_weight}. Because the $p_i$ are pairwise distinct and $V$ must vanish to order exactly $d_i$ at the $p_i$, we may regard $(V,p_1,p_2,p_3,p_4)$ as arising locally from the intersections of $\pr_{i}^{*}[\Sigma_{d_i-k_i-1,k_i}]$, with multiplicities as dictated by Lemma \ref{principal_parts}. Therefore, it suffices to prove that the $\pr_{i}^{*}\Sigma_{d_i-k_i-1,k_i}$ intersect transversely in $G\times E_1\times E_2\times E_3\times E_4$.

Suppose that this were not the case. We would then have a non-trivial deformation $(\wt{V},\wt{p}_1,\wt{p}_2,\wt{p}_3,\wt{p}_4)$ of $(V,p_1,p_2,p_3,p_4)$. Letting $B=\Spec k[\epsilon]/\epsilon^2$, this means explicitly that $\wt{V}\subset H^0(E,\cL)\otimes_k B$ is a linear system on $E\times B$ with underlying line bundle $\cL_{B}=\cL\otimes_k B$, and has vanishing $(k_i,d_i-k_i)$ along sections $\wt{p}_1,\wt{p}_2,\wt{p}_3,\wt{p}_4$ of $E\times B\to B$ restricting to $p_1,p_2,p_3,p_4$.

We now remove the base-points of $\wt{V}$ by twisting, and apply a translation so that $\wt{p}_i$ becomes the identity section. Let $\tau:E\times B\to E\times B$ be the translation by $\wt{p}_1$. Explicitly, have a new quintuple $(\wt{V}',\wt{p},\wt{p}'_2,\wt{p}'_3,\wt{p}'_4)$ with
\begin{align*}
\wt{V}'&=\tau^{*}\left(\wt{V}\left(-\sum_i k_i\wt{p}_i\right)\right)\\
\wt{p}'_i&=\tau^{*}(\wt{p}_i);
\end{align*}
in particular, $\wt{p}'_1$ is just the identity section $\wt{p}$. Note that the underlying line bundle of $\wt{V}'$ is
\begin{equation*}
\cL'=\tau^{*}\left(\cL_B\left(-\sum_i k_i\wt{p}_i\right)\right).
\end{equation*}

Let $\cH$ be the Hurwitz space parametrizing degree $d-\sum_{i=1}^4 k_i$ covers $f:X\to\bP^1$ ramified to order $d_i-2k_i-1$ at pairwise distinct marked points $p_i\in X$ for $i=1,2,3,4$, where $X$ is a smooth curve of genus 1, and let $\psi:\cH\to\cM_{1,1}$ be the map remembering the elliptic curve $(X,p_1)$. We claim that $\wt{V}'$ gives rise to a non-trivial tangent vector $v$ of $\cH$ in the kernel of $d\psi$. It suffices to prove that, with $E$ fixed, we can recover the deformation of $(V,p_1,p_2,p_3,p_4)$ from the data of $(\wt{V}',\wt{p},\wt{p}'_2,\wt{p}'_3,\wt{p}'_4)$. Indeed, we have
\begin{equation*}
\tau^{*}\cL_B=\cL'\left(\sum_i k_i\wt{p}'_i\right).
\end{equation*}
We may recover the translation $\tau$, by the \'{e}taleness of the group scheme $K(\cL_B)$ over $B$, and thus the section $\wt{p}_1$. Now, by inverting the formulas for $\wt{P}'$ and $\wt{p}'_i$, we may recover $\wt{P}$ and $\wt{p}_i$ as well.

Finally, $\cH$ and $\cM_{1,1}$ are smooth, hence the map $\cH\to\cM_{1,1}$ is generically smooth. Thus, $v$ can only map to special $(E,p)\in\cM_{1,1}$. Because $E$ is general, we have reached a contradiction, completing the proof of the lemma.
\end{proof}

\vskip 3pt

\subsection{Proof of Theorem \ref{genus1_weighted_count}}

By \ref{class_to_integrate} and Lemma \ref{identify_pencil_weights}, to prove Theorem \ref{genus1_weighted_count} in the case $d_i\ge2$, we need to compute the integral of the class
\begin{dmath}
\pr_{1}^{*}[\Sigma_1]^{d_1-1}\cdot(\pr_{2}^{*}[\Sigma_1]-(d_1-1)[\Delta_{12}])^{d_2-1}\cdot(\pr_{3}^{*}[\Sigma_1]-(d_1-1)[\Delta_{13}]-(d_2-1)[\Delta_{23}])^{d_3-1}\cdot(\pr_4^{*}[\Sigma_1]-(d_1-1)[\Delta_{14}]-(d_2-1)[\Delta_{24}]-(d_3-1)[\Delta_{34}])^{d_4-1}
\end{dmath}
on $G\times E_1\times E_2\times E_3\times E_4$. It suffices to work in numerical equivalence; we will do so throughout this section.

\begin{lem}
In $\Num(G\times E_i)$, we have
\begin{equation*}
    [\Sigma_1]=\sigma_1+2dx_i,
\end{equation*}
where $\sigma_1\in \Num(G)$ is the usual Schubert cycle and $x_i\in \Num(E_i)$ is the class of a point.
\end{lem}

\begin{proof}
We first compute the classes of $c(\cF_k)$ and $c(\cV)$, as defined in \S \ref{ram_loci_section}. Because $\cI_{\Delta}/\cI_{\Delta}^2\cong N^{\vee}_{E_i/E_i\times E_i}$ is trivial, we may filter $\cO_{E_i\times E_i}/\cI_{\Delta}^{k}$ by $k$ trivial line bundle quotients on $E$, and thus
\begin{equation*}
c(\cF_k)=(1+c_1(\cL))^k=1+dkx_i.
\end{equation*}
Next, let $\cV\to H^0(E,\cL)\otimes_k \cO_G$ be the tautological inclusion. Let $W\subset H^0(E,\cL)$ be a subspace of codimension $k+1$. By definition, the first degeneracy locus of the composition $\cV\to (H^0(E,\cL)/W)\otimes_k \cO_G$ is $\sigma_k$, and by Porteous, its class is also equal to $\{c(\cV)^{-1}\}_k$. We thus conclude that
\begin{equation*}
\frac{1}{c(\cV)}=\sum_{i=0}^{d-2}\sigma_i.
\end{equation*}
Now, $\Sigma_{k,0}$ is the first degeneracy locus of $\varphi_{k+1}$, so by Porteous, we have
\begin{equation*}
[\Sigma_{k,0}]=\left\{c(\cF_{k+1})\cdot\frac{1}{c(\cV)}\right\}_{k}=\sigma_{k}+d(k+1)\sigma_{k-1}x_i.
\end{equation*}
\end{proof}

We thus have
\begin{equation*}
    \widetilde{N}_{d_1,d_2,d_3,d_4}=\frac{1}{d^2}\int_{G\times E_1\times E_2\times E_3\times E_4}R_1R_2R_3R_4,
\end{equation*}
where
\begin{align*}
    R_1&=(\sigma_1+2dx_1)^{d_1-1},\\
    R_2&=(\sigma_1+2dx_2-(d_1-1)\Delta_{12})^{d_2-1},\\
    R_3&=(\sigma_1+2dx_3-(d_1-1)\Delta_{13}-(d_2-1)\Delta_{23})^{d_3-1},\\
    R_4&=(\sigma_1+2dx_4-(d_1-1)\Delta_{14}-(d_2-1)\Delta_{24}-(d_3-1)\Delta_{34})^{d_4-1}.
\end{align*}
Here, all classes are regarded as pulled back to the ambient space. We have
\begin{align*}
    R_2&=\quad\sigma_1^{d_2-1}\\
    &\quad+\sigma_1^{d_2-2}(d_2-1)(2dx_2-(d_1-1))\Delta_{12}\\
    &\quad+\sigma_1^{d_2-3}(d_2-1)(d_2-2)\cdot(-2d(d_1-1)x_1x_2)
\end{align*}
and
\begin{align*}
    R_3&=\quad\sigma_1^{d_3-1}\\
    &\quad+\sigma_{1}^{d_3-2}(d_3-1)(2dx_3-(d_1-1)\Delta_{13}-(d_2-1)\Delta_{23})\\
    &\quad+\sigma_{1}^{d_3-3}(d_3-1)(d_3-2)(-2d(d_1-1)x_1x_3-2d(d_2-1)x_2x_3+(d_1-1)(d_2-1)\Delta_{123})\\
    &\quad+\sigma_{1}^{d_3-4}(d_3-1)(d_3-2)(d_3-3)\cdot2d(d_1-1)(d_2-1)x_1x_2x_3
\end{align*}
Multiplying,
\begin{align*}
    R_2R_3&=\quad\sigma_{1}^{d_2+d_3-2}\\
    &\quad+\sigma_{1}^{d_2+d_3-3}(2d(d_2-1)x_2+2d(d_3-1)x_3\\
    &\qquad\qquad-(d_1-1)(d_2-1)\Delta_{12}-(d_1-1)(d_3-1)\Delta_{13}-(d_2-1)(d_3-1)\Delta_{23})\\
    &\quad+\sigma_{1}^{d_2+d_3-4}(2d(d_1-1)(d_2-1)(d_2-2)x_1x_2-2d(d_1-1)(d_3-1)(d_3-2)x_1x_3\\
    &\qquad\qquad+2d(d_2-1)(d_3-1)(2d-d_2-d_3+3)x_2x_3\\
    &\qquad\qquad-2d(d_1-1)(d_2-1)(d_3-1)\Delta_{12}x_3-2d(d_1-1)(d_2-1)(d_3-1)\Delta_{13}x_2\\
    &\qquad\qquad+(d_1-1)(d_2-1)(d_3-1)(d_1+d_2+d_3-4)\Delta_{123})\\
    &\quad+\sigma_1^{d_2+d_3-5}(-2d(d_1-1)(d_2-1)(d_3-1)d_4(d_2+d_3-4)x_1x_2x_3).
\end{align*}

We next multiply with
\begin{align*}
    R_4&=\quad\sigma_1^{d_4-1}\\
    &\quad+\sigma_1^{d_4-2}(d_4-1)(2dx_4-(d_1-1)\Delta_{14}-(d_2-1)\Delta_{24}-(d_3-1)\Delta_{34})\\
    &\quad+\sigma_1^{d_4-3}(d_4-1)(d_4-2)(-2d(d_1-1)x_1x_4-2d(d_2-1)x_2x_4-2d(d_3-1)x_3x_4\\
    &\qquad\qquad+(d_1-1)(d_2-1)\Delta_{124}+(d_1-1)(d_3-1)\Delta_{134}+(d_2-1)(d_3-1)\Delta_{234})\\
    &\quad+\sigma_{1}^{d_4-4}(d_4-1)(d_4-2)(d_4-3)(2d(d_1-1)(d_2-1)x_1x_2x_4+2d(d_1-1)(d_3-1)x_1x_3x_4\\
    &\qquad\qquad+2d(d_2-1)(d_3-1)x_2x_3x_4-(d_1-1)(d_2-1)(d_3-1)\Delta_{1234})\\
    &\quad+\sigma_{1}^{d_4-5}(d_4-1)(d_4-2)(d_4-3)(d_4-4)(-2d(d_1-1)(d_2-1)(d_3-1)x_1x_2x_3x_4).
\end{align*}
In the product $R_2R_3R_4$, we only wish to extract the terms that will be non-zero after multiplying by
\begin{align*}
    R_1&=\sigma_1^{d_1-1}+\sigma_1^{d_1-2}(d_1-1)(2dx_1)
\end{align*}
and integrating. These are the terms of $R_2R_3R_4$ that have factors of exactly $\sigma_1^{d_2+d_3+d_4-6}$ and $\sigma_1^{d_2+d_3+d_4-7}$.

First, we extract the terms having a factor of $\sigma_1^{d_2+d_3+d_4-6}$, and multiply by $x_1$. There are three non-zero contributions: the product of the $\sigma_{1}^{d_2+d_3-i}$ term of $R_2R_3$ and the $\sigma_1^{d_4-(6-i)}$ for $i=2,3,4$ (when $i=5$, multiplying by $x_1$ kills the term coming from $R_2R_3$). These are listed below; we suppress the factor of $\sigma_{1}^{d_2+d_3+d_4-6}x_1x_2x_3x_4$ appearing in all three.
\begin{itemize}
    \item $i=2$: $(d_2-1)(d_3-1)(d_4-1)(d_4-2)(d_4-3)(d_2+d_3+d_4-3)$
    \item $i=3$: $-2(d_2-1)(d_3-1)(d_4-1)^2(d_4-2)(d_2+d_3+d_4-3)$
    \item $i=4$: $(d_2-1)(d_3-1)(d_4-1)^2d_4(d_2+d_3+d_4-3)$
\end{itemize}
The sum of these contributions is
\begin{align*}
    &\quad(d_2-1)(d_3-1)(d_4-1)(d_2+d_3+d_4-3)((d_4-2)(d_4-3)-2(d_4-1)(d_4-2)+d_4(d_4-1))\\
    &=2(d_2-1)(d_3-1)(d_4-1)(d_2+d_3+d_4-3).
\end{align*}
Therefore, we get a total contribution to $\widetilde{N}_{d_1,d_2,d_3,d_4}$ of
\begin{equation*}
    4d(d_1-1)(d_2-1)(d_3-1)(d_4-1)(d_2+d_3+d_4-3)\cdot\int_{G}\sigma_1^{2d-4}.
\end{equation*}

Now, we consider the contribution from terms with a factor of $\sigma_1^{d_2+d_3+d_4-7}$. Here, there are four non-zero contributions: the product of the $\sigma_{1}^{d_2+d_3-i}$ term of $R_2R_3$ and the $\sigma_1^{d_4-(7-i)}$ for $i=2,3,4,5$. Suppressing the factors of $\sigma_{1}^{d_2+d_3+d_4-7}x_1x_2x_3x_4$, they are:
\begin{itemize}
    \item $i=2$: $-2d(d_1-1)(d_2-1)(d_3-1)(d_4-1)(d_4-2)(d_4-3)(d_4-4)$
    \item $i=3$: $2d(d_1-1)(d_2-1)(d_3-1)(d_4-1)(d_4-2)(d_4-3)(2d_4-d_2-d_3)$
    \item $i=4$: $2d(d_1-1)(d_2-1)(d_3-1)(d_4-1)^2(d_4-2)(2d_2+2d_3-d_4-6)$
    \item $i=5$: $-2d(d_1-1)(d_2-1)(d_3-1)(d_4-1)^2d_4(d_2+d_3-4)$
\end{itemize}
The sum of these contributions is
\begin{align*}
    &\quad2d(d_1-1)(d_2-1)(d_3-1)(d_4-1)(-(d_4-2)(d_4-3)(d_4-4)\\
    &\qquad\qquad+(d_4-2)(d_4-3)(2d_4-d_2-d_3)\\
    &\qquad\qquad+(d_4-1)(d_4-2)(2d_2+2d_3+d_4-6)+(d_4-1)(d_2+d_3-4))\\
    &=-4d(d_1-1)(d_2-1)(d_3-1)(d_4-1)(d_2+d_3+d_4-6).
\end{align*}
The corresponding contribution to $\widetilde{N}_{d_1,d_2,d_3,d_4}$ is then
\begin{equation*}
    -4d(d_1-1)(d_2-1)(d_3-1)(d_4-1)(d_2+d_3+d_4-6)\cdot\int_{G}\sigma_1^{2d-4}.
\end{equation*}

Summing the two contributions, we conclude:

\begin{equation}\label{tildeN_final_formula}
    d^2\widetilde{N}_{d_1,d_2,d_3,d_4}=12d(d_1-1)(d_2-1)(d_3-1)(d_4-1)C_{d-2},
\end{equation}

\begin{proof}[Proof of Theorem \ref{genus1_weighted_count}]
When $d_1,d_2,d_3,d_4\ge 2$, we have (\ref{tildeN_final_formula}). When $d_i=1$, the right hand side of (\ref{tildeN_final_formula}) becomes zero, and indeed, Proposition \ref{expected_dim} implies that $\wt{N}_{d_1,d_2,d_3,d_4}=0$.
\end{proof}

\vskip 3pt

\subsection{Variants}

Here, we make some auxiliary remarks on variants of the method of computation above. First, note that in the first step of the proof, instead of imposing the condition $T_{d_1-1}\subset G\times E$, we could have directly imposed the condition that $V$ has \textit{ramification index} at least $d_1$ at $p_1$, that is, computed the locus $\Sigma_{d_1-1}\subset G\times E$ by Porteous's formula. Then, as we need to subtract excess loci in the subsequent steps, the remainder of the computation will remain the same. Carrying out the computation in this way yields the following:

\begin{prop}\label{first_point_unweighted}
Let $(E,p_1),d,d_1,d_2,d_3,d_4$ be as above. Then, the weighted number $\wt{N}^{\circ}_{d_1,d_2,d_3,d_4}$ of tuples $(V,p_2,p_3,p_4)$ of pencils with vanishing $(0,d_1)$ at $p_1$ and total vanishing $d_i$ at $p_i$ for $i=2,3,4$ is
\begin{equation*}
\wt{N}^{\circ}_{d_1,d_2,d_3,d_4}=2d_1(d_1+1)(d_1-1)(d_2-1)(d_3-1)(d_4-1)\binom{2d-d_1-2}{d-d_1}\cdot\frac{1}{d(d-1)}.
\end{equation*}
Here, the multiplicity of $(V,p_2,p_3,p_4)$ in the weighted count is
\begin{equation*}
C^{d_2,d_3,d_4}_{k_2,k_3,k_4}=\prod_{i=2}^{4}c_{d_i-k_i-i,k_i},
\end{equation*}
where $(k_i,d_i-k_i)$ is the vanishing sequence of $V$ at $p_i$.
\end{prop}

Instead of working on $G\times E_1\times E_2\times E_3\times E_4$, one can also prove Theorem \ref{genus1_weighted_count} via an analogous computation on the smooth moduli variety $G\times\overline{\cM}_{E,5}$, where $\overline{\cM}_{E,5}$ denotes the fiber of the forgetful map $\overline{\cM}_{1,5}\to\overline{\cM}_{1,1}$ over $(E,p_1)$. The ramification loci may then be expressed in terms of tautological classes on $\overline{\cM}_{1,5}$. One pleasant feature is that the analogues of ``diagonal'' loci $\Delta_{ij}$ appear with multiplicity 1, and in all of the classes $\pr_i^{*}\Sigma_1$ (not just those with $j<i$), so the class of $T_{d_1,d_2,d_3,d_4}$ in this setting is clearly symmetric under permutation of the $p_i$. 

In fact, in this setting, it is natural to perform the computation in smooth families, for instance, over the universal family $\cC_{1,1}\to\cM_{1,1}$ of elliptic curves. While it would be desirable for the method to extend further to the singular fiber in the family $\overline{\cM}_{1,2}\to\overline{\cM}_{1,1}$, for instance, to compute certain pure-cycle Hurwitz numbers, it breaks down at singular points. 

In either setting, one can extend the technique to higher genus curves $C$, and allow the line bundle $\cL$ to vary. For example, let $J=\Pic^d(C)$, and assume for simplicity that $d>2g-2$. Let $\pr_J:C\times J\to J$ be the projection map, and let $\cE=(\pr_J)_{*}\cP$, where $\cP$ is the Poincaré bundle. Then, one can define the ramification loci as before on $\Gr(2,\cE)\times C^{3g}$. There are no obstructions to generalizing Theorem \ref{genus1_weighted_count} to higher genus except for the combinatorial difficulty of having $3g$ copies of $C$. Thus, to answer Question \ref{main_question} in the case $g>1$, we instead use the degeneration approach in \S \ref{degeneration_section}.

Finally, the method we have developed also works in enumerating higher rank linear systems, with two caveats. First, as Proposition \ref{expected_dim} fails in higher rank, it is necessary to restrict to cases in which the expected dimension statements are guaranteed to hold, for instance, if $m$ is small (see \cite{eh_counterexamples}, \cite{edidin}, and \cite{farkas_bn}). Second, one can again obtain counts of linear systems with imposed conditions of \textit{total vanishing}, but unlike in rank 1, it is not possible to recover the counts of linear systems with prescribed \textit{vanishing sequences} simply by twisting away base-points. Thus, results such as that of Farkas-Tarasca \cite{fartar} remain out of reach of our techniques.

\vskip 3pt

\section{Base-point-free pencils in genus 1: Proof of Theorems \ref{genus1_unweighted_count} and \ref{genus1_duality}}\label{unweighted_section}

As in the previous section, let $(E,p_1)$ be a general elliptic curve, let $d,d_1,d_2,d_3,d_4$ be integers such that $1\le d_i\le d$ and $d_1+d_2+d_3+d_4=2d+4$. Let $N_{d_1,d_2,d_3,d_4}$ be the number of 4-tuples $(p_2,p_3,p_4,f)$, where $p_i\in E$ are pairwise distinct points, and $f:E\to\bP^1$ is a morphism of degree $d$ with ramification index at least (and, by Corollary \ref{enumerated_covers_are_generic}, exactly) $d_i$ at each $p_i$. In this section, we use the fact that the $N_{d_1,d_2,d_3,d_4}$ are determined by the $\wt{N}_{d_1,d_2,d_3,d_4}$ to prove Theorems \ref{genus1_unweighted_count} and \ref{genus1_duality}.

\begin{prop}\label{recursion}
We have:
\begin{equation*}
\wt{N}_{d_1,d_2,d_3,d_4}\\
=\sum_{k_1,k_2,k_3,k_4\ge0}C^{d_1,d_2,d_3,d_4}_{k_1,k_2,k_3,k_4}N_{d_1-2k_1,d_2-2k_2,d_3-2k_3,d_4-2k_4}.
\end{equation*}
\end{prop}

\begin{proof}
After adding base-points of order $k_i$, each term on the left hand side counts the number of pencils of degree $d$ on $E$ with vanishing sequence $(k_i,d_i-k_i)$ at $p_i$ for $i=1,2,3,4$, with the appropriate multiplicity as in the definition of $\wt{N}_{d_1,d_2,d_3,d_4}$.
\end{proof}

\subsection{Generating functions}

It is natural to package the numbers $\wt{N}_{d_1,d_2,d_3,d_4},N_{d_1,d_2,d_3,d_4}$ into generating functions; after doing so, we obtain a formula for $N_{d_1,d_2,d_3,d_4}$ as a particular coefficient in a power series in one variable, see Proposition \ref{N_coeff_onevar}. It will be convenient to treat $d$ as an independent variable from the $d_i$, so we first extend the definitions of $N_{d_1,d_2,d_3,d_4}$ and $\wt{N}_{d_1,d_2,d_3,d_4}$. 

\begin{defn}\label{extended_N}
For any integers $d,d_1,d_2,d_3,d_4$ with $d_i\ge 1$ and $d\ge2$, define
\begin{equation*}
    \widetilde{N}^d_{d_1,d_2,d_3,d_4}=\frac{12C_{d-2}}{d}(d_1-1)(d_2-1)(d_3-1)(d_4-1).
\end{equation*}
Then, define $N^d_{d_1,d_2,d_3,d_4}$ inductively as the unique integers satisfying 
\begin{equation*}
\wt{N}^d_{d_1,d_2,d_3,d_4}\\
=\sum_{k_1,k_2,k_3,k_4\ge0}C^{d_1,d_2,d_3,d_4}_{k_1,k_2,k_3,k_4}N^d_{d_1-2k_1,d_2-2k_2,d_3-2k_3,d_4-2k_4}
\end{equation*}
for all $d_i\ge1,d\ge2$.
\end{defn}

Clearly, when $d_1+d_2+d_3+d_4=2d+4$, we have $\wt{N}_{d_1,d_2,d_3,d_4}=\wt{N}^d_{d_1,d_2,d_3,d_4}$ and $N_{d_1,d_2,d_3,d_4}=N^d_{d_1,d_2,d_3,d_4}$.

\begin{defn}\label{N_generating_functions}
Define the generating functions
\begin{align*}
N(x_1,x_2,x_3,x_4,q)&=\sum_{d\ge2,d_i\ge1}N^d_{d_1,d_2,d_3,d_4}x_1^{d_1}x_2^{d_2}x_3^{d_3}x_4^{d_4}q^{d}\\
\wt{N}(x_1,x_2,x_3,x_4,q)&=\sum_{d\ge2,d_i\ge1}\wt{N}^d_{d_1,d_2,d_3,d_4}x_1^{d_1}x_2^{d_2}x_3^{d_3}x_4^{d_4}q^{d}
\end{align*}
\end{defn}

\begin{lem}\label{tildeN_gf_formula}
We have
\begin{equation*}
\wt{N}(x_1,x_2,x_3,x_4,q)=[(6q-1)+(1-4q)^{3/2}]\cdot\prod_{i=1}^{4}\left(\frac{x_i}{1-x_i}\right)^2
\end{equation*}
\end{lem}
\begin{proof}
Indeed,
\begin{align*}
&\quad\wt{N}(x_1,x_2,x_3,x_4,q)\\
&=\sum_{d_i,d}\frac{12}{d}(d_1-1)(d_2-1)(d_3-1)(d_4-1)C_{d-2}x_1^{d_1}x_2^{d_2}x_3^{d_3}x_4^{d_4}q^{d}\\
&=\sum_{d=2}^{\infty}\frac{12C_{d-2}}{d}q^d\cdot\prod_{i=1}^4\left(\sum_{d_i=1}^{\infty}(d_i-1)x_i^{d_i}\right)\\
&=[(6q-1)+(1-4q)^{3/2}]\cdot\prod_{i=1}^{4}\left(\frac{x_i}{1-x_i}\right)^2,
\end{align*}
where in the last step we obtain the generating function for the sequence $b_d=12C_{d-2}/d$ by integrating that of the Catalan numbers, see \cite[Example 6.2.6]{stanley}.
\end{proof}

\begin{prop}\label{N_vs_tildeN}
The generating functions $N(x_1,x_2,x_3,x_4,q)$ and $\wt{N}(x_1,x_2,x_3,x_4,q)$ are related by the following formulas:
\begin{align*}
\wt{N}(x_1,x_2,x_3,x_4,q)&=N\left(\frac{1-\sqrt{1-4x_{1}^2q}}{2x_{1}q},\frac{1-\sqrt{1-4x_{2}^2q}}{2x_{2}q},\frac{1-\sqrt{1-4x_{3}^2q}}{2x_{3}q},\frac{1-\sqrt{1-4x_{4}^2q}}{2x_{4}q},q\right)\\
N(x_1,x_2,x_3,x_4,q)&=\wt{N}\left(\frac{x_1}{1+x_1^2q},\frac{x_2}{1+x_2^2q},\frac{x_3}{1+x_3^2q},\frac{x_4}{1+x_4^2q},q\right)
\end{align*}
\end{prop}

\begin{proof}
The second formula will follow directly from the first. We have
\begin{align*}
&\quad\wt{N}(x_1,x_2,x_3,x_4,q)\\
&=\sum_{d,d_i}\wt{N}^d_{d_1,d_2,d_3,d_4}x_1^{d_1}x_2^{d_2}x_3^{d_3}x_4^{d_4}q^{d}\\
&=\sum_{d,d_i,k_i} c_{d_1-k_1-1,d_1}\cdots c_{d_4-k_4-1,d_4}\cdot N^{d-k_1-k_2-k_3-k_4}_{d_1-2k_1,d_2-2k_2,d_3-2k_3,d_4-2k_4}x_1^{d_1}x_2^{d_2}x_3^{d_3}x_4^{d_4}q^{d}\\
&=\sum_{d,d_i}\left[x_1^{d_1}x_2^{d_2}x_3^{d_3}x_4^{d_4}q^{d}\prod_{i=1}^4\left(\sum_{m_i=0}^{\infty}c_{d_i+m_{i}-1,m_i}(x_i^2q)^{m_i}\right)\right]N^d_{d_1,d_2,d_3,d_4}\\
&=\sum_{d,d_i}\left[x_1^{d_1}x_2^{d_2}x_3^{d_3}x_4^{d_4}q^{d}\prod_{i=1}^4\left(\frac{1-\sqrt{1-4x_{i}^2q}}{2x_{i}^2q}\right)^{d_i}\right]N^d_{d_1,d_2,d_3,d_4}\\
&=\sum_{d,d_i}\left[q^{d}\prod_{i=1}^4\left(\frac{1-\sqrt{1-4x_{i}^2q}}{2x_{i}q}\right)^{d_i}\right]N^d_{d_1,d_2,d_3,d_4}\\
&=N\left(\frac{1-\sqrt{1-4x_{1}^2q}}{2x_{1}q},\frac{1-\sqrt{1-4x_{2}^2q}}{2x_{2}q},\frac{1-\sqrt{1-4x_{3}^2q}}{2x_{3}q},1-\frac{\sqrt{1-4x_{4}^2q}}{2x_{4}q},q\right),
\end{align*}
where in the fifth line we have applied Lemma \ref{catalan_gf}.
\end{proof}

Combining Lemma \ref{tildeN_gf_formula} and the second part of Proposition \ref{N_vs_tildeN}, we obtain:

\begin{cor}\label{N_gf_formula}
We have
\begin{dmath}\label{N_gf_formula_eq}
N(x_1,x_2,x_3,x_4,q)=[(6q-1)+(1-4q)\sqrt{1-4q}]\cdot\prod_{i=1}^{4}\left(\frac{x_i}{1-x_i+x_i^2q}\right)^2
\end{dmath}
\end{cor}

\begin{lem}\label{rational_function_x_coeffs}
We have
\begin{dmath*}
\left(\frac{x}{1-x+x^2q}\right)^2=\sum_{n=0}^{\infty}\frac{1}{1-4q}\left(\sum_{k+\ell=n}\frac{\alpha^k-\beta^k}{\alpha-\beta}\cdot\frac{\alpha^\ell-\beta^\ell}{\alpha-\beta}\right)x^n
\end{dmath*}
where
\begin{align*}
\alpha&=\frac{1+\sqrt{1-4q}}{2},\\
\beta&=\frac{1-\sqrt{1-4q}}{2}.
\end{align*}
Furthermore, the coefficient of $x^n$ above is a polynomial in $q$ of degree $\displaystyle\left\lfloor\frac{n}{2}\right\rfloor-1$.
\end{lem}

\begin{proof}
One verifies by a straightforward computation that
\begin{align*}
&\left(\frac{x}{1-x+x^2q}\right)^2\\
=&\frac{1}{1-4q}\left[\left(\frac{1}{1-\alpha x}\right)^2+\left(\frac{1}{1-\beta x}\right)^2-\left(\frac{1}{1-\alpha x}+\frac{1}{1-\beta x}\right)-\frac{1}{\sqrt{1-4q}}\left(\frac{1}{1-\alpha x}-\frac{1}{1-\beta x}\right)\right]\\
=&\sum_{n=0}^{\infty}\frac{1}{1-4q}\left[n(\alpha^n+\beta^n)-\frac{1}{\sqrt{1-4q}}(\alpha^n-\beta^n)\right]x^n.
\end{align*}
Then, note that
\begin{align*}
&\frac{1}{1-4q}\left[n(\alpha^n+\beta^n)-\frac{1}{\sqrt{1-4q}}(\alpha^n-\beta^n)\right]\\
=&\frac{1}{(\alpha-\beta)^2}\left[n(\alpha^n+\beta^n)-(\alpha+\beta)\cdot\frac{\alpha^n-\beta^n}{\alpha-\beta}\right]\\
=&\sum_{k+\ell=n}\frac{\alpha^k-\beta^k}{\alpha-\beta}\cdot\frac{\alpha^\ell-\beta^\ell}{\alpha-\beta},
\end{align*}
which is a symmetric polynomial of degree $n-2$ in $\alpha$ and $\beta$. Because $\alpha+\beta=1$ and $\alpha\beta=1-4q$, the coefficient of $x^n$ is thus a polynomial of degree
\begin{equation*}
\left\lfloor\frac{n-2}{2}\right\rfloor=\left\lfloor\frac{n}{2}\right\rfloor-1
\end{equation*}
in $q$, as claimed.
\end{proof}

\begin{cor}\label{N_coeff_onevar}
We have
\begin{equation*}
N_{d_1,d_2,d_3,d_4}=(1-4q)^{3/2}\cdot\prod_{i=1}^{4}\left(\sum_{j=0}^{d_1-2}s_{j}(\alpha,\beta)s_{d_i-2-j}(\alpha,\beta)\right)[q^d],
\end{equation*}
where 
\begin{equation*}
s_{j}(x,y)=\frac{x^{j+1}-y^{j+1}}{x-y}
\end{equation*}
is a Schur polynomial in two variables. 
\end{cor}

\begin{proof}
This is a consequence of Corollary \ref{N_gf_formula} and Lemma \ref{rational_function_x_coeffs}. We may ignore the contribution of the $(6q-1)$ term appearing on the right hand side of (\ref{N_gf_formula_eq}), because, by the last statement in Lemma \ref{rational_function_x_coeffs}, the degree of the coefficient of $x_1^{d_1}x_2^{d_2}x_3^{d_3}x_4^{d_4}$ in
\begin{equation*}
\prod_{i=1}^{4}\left(\frac{x_i}{1-x_i+x_i^2q}\right)^2
\end{equation*}
as a polynomial in $q$ is
\begin{equation*}
\sum_{i=1}^{4}\left(\left\lfloor\frac{d_i}{2}\right\rfloor-1\right)\le d-2,
\end{equation*}
and thus contributes nothing to the $q^d$ coefficient after multiplication by $(6q-1)$.
\end{proof}

\vskip 3pt

\subsection{Schubert cycle formula}

We now relate the formula in Corollary \ref{N_coeff_onevar} to intersection numbers on the Grassmannian to prove Theorem \ref{genus1_unweighted_count}(a).

\begin{lem}\label{integral_formula}
Let $d$ be a positive integer, and let $f(x,y)$ be a homogeneous symmetric polynomial with $\deg(f)\le 2d-2$. Then, we have
\begin{equation*}
\left(-\frac{1}{2}(1-4q)^{1/2}\cdot f(\alpha,\beta)\right)[q^d]=\int_{\Gr(2,d+1)}f(x,y)\cdot(x+y)^{2d-2-\deg(f)},
\end{equation*}
where as before we put
\begin{align*}
\alpha&=\frac{1+\sqrt{1-4q}}{2},\\
\beta&=\frac{1-\sqrt{1-4q}}{2},
\end{align*}
and the integrand on the right hand side is viewed as a top cohomology class on $\Gr(2,d+1)$ via the identification of Schur polynomials $s_j$ and Schubert cycles $\sigma_j$.
\end{lem}

\begin{proof}
The vector space of symmetric polynomials $f(x,y)$ is spanned by polynomials of the form 
\begin{equation*}
f(x,y)=(xy)^m(x+y)^n=s_{11}(x,y)^m\cdot s_1(x,y)^n,
\end{equation*}
where $2m+n=2d-2$; it suffices to prove the claim for such $f$. Note that $f(\alpha,\beta)=q^m$. Now,
\begin{align*}
\int_{\Gr(2,d+1)}f(x,y)\cdot(x+y)^{2d-2-\deg(f)}&=\int_{\Gr(2,d+1)}(xy)^m\cdot(x+y)^{2d-2-2m}\\
&=\int_{\Gr(2,d+1)}\sigma_{11}^m\cdot\sigma_1^{2d-2-2m}\\
&=C_{d-m-1}\\
&=-\frac{1}{2}(1-4q)^{1/2}[q^{d-m}]\\
&=-\frac{1}{2}f(\alpha,\beta)(1-4q)^{1/2}[q^{d-m}],
\end{align*}
where we have applied the Pieri Rule and Lemma \ref{catalan_gf}.
\end{proof}

\begin{prop}\label{N_schubert_formula}
We have
\begin{equation*}
N_{d_1,d_2,d_3,d_4}=\int_{\Gr(2,d+1)}\left(\prod_{i=1}^{4}\sum_{a_i+b_i=d_i-2}\sigma_{a_i}\sigma_{b_i}\right)(8\sigma_{11}-2\sigma_1^2).
\end{equation*}
\end{prop}

\begin{proof}
Applying Corollary \ref{N_coeff_onevar},
\begin{align*}
N_{d_1,d_2,d_3,d_4}&=(1-4q)^{3/2}\prod_{i=1}^{4}\left(\sum_{j=0}^{d_1-2}s_{j}(\alpha,\beta)s_{d_i-2-j}(\alpha,\beta)\right)[q^d]\\
&=(-2+8q)\cdot\left[-\frac{1}{2}(1-4q)^{1/2}\prod_{i=1}^{4}\left(\sum_{j=0}^{d_1-2}s_{j}(\alpha,\beta)s_{d_i-2-j}(\alpha,\beta)\right)\right][q^d]\\
&=8\cdot\left[-\frac{1}{2}(1-4q)^{1/2}\prod_{i=1}^{4}\left(\sum_{j=0}^{d_1-2}s_{j}(\alpha,\beta)s_{d_j-2-j}(\alpha,\beta)\right)\right][q^{d-1}]\\
&\quad-2\cdot\left[-\frac{1}{2}(1-4q)^{1/2}\prod_{i=1}^{4}\left(\sum_{j=0}^{d_1-2}s_{j}(\alpha,\beta)s_{d_i-2-j}(\alpha,\beta)\right)\right][q^d]\\
&=8\int_{\Gr(2,d)}\left(\prod_{i=1}^{4}\sum_{a_i+b_i=d_i-2}\sigma_{a_i}\sigma_{b_i}\right)-2\int_{\Gr(2,d+1)}\left(\prod_{i=1}^{4}\sum_{a_i+b_i=d_i-2}\sigma_{a_i}\sigma_{b_i}\right)\sigma_1^2\\
&=\int_{\Gr(2,d+1)}\left(\prod_{i=1}^{4}\sum_{a_i+b_i=d_i-2}\sigma_{a_i}\sigma_{b_i}\right)(8\sigma_{11}-2\sigma_1^2),
\end{align*}
where in the second to last step we have applied Lemma \ref{integral_formula} to the polynomial
\begin{equation*}
f(x,y)=\prod_{i=1}^{4}\sum_{j=0}^{d_i-2}s_j(x,y)s_{d_i-2-j}(x,y)
\end{equation*}
of degree $(d_1-2)+\cdots+(d_4-2)=2(d-1)-2$.
\end{proof}

\vskip 3pt

\subsection{Laurent polynomial formula}
Here, we expand the formula in Proposition \ref{N_schubert_formula} to prove Theorem \ref{genus1_unweighted_count}(b).

\begin{lem}\label{integral_fourfold_product}
Let $n_1, n_2, n_3, n_4$ be integers satisfying $0\le n_i\le d-1$ and $n_1+n_2+n_3+n_4=2d-4$. Then, we have:
\begin{equation*}
\int_{\Gr(2,d)}\sigma_{n_1}\sigma_{n_2}\sigma_{n_3}\sigma_{n_4}=\min(d-n_1-1,n_4+1)
\end{equation*}
\end{lem}

\begin{proof}
Without loss of generality, suppose that $n_1\ge n_2\ge n_3\ge n_4$, so that $n_1+n_2\ge d-2$. If $n_1=d_1-1$, then $\sigma_{n_1}=0=\min(d-n_1-1,n_4+1)$, so assume that $n_1\le d_1-2$. By the Pieri Rule, we have
\begin{equation*}
\sigma_{n_1}\sigma_{n_2}=\sigma_{n_1,n_2}+\sigma_{n_1+1,n_2-1}+\cdots+\sigma_{d-2,n_1+n_2-d+2}.
\end{equation*}
We wish to express the product of this class with $\sigma_{n_3}$ in the Schubert cycle basis and extract the coefficient of $\sigma_{d-2,d-2-n_{4}}$. By the Pieri rule, each product $\sigma_{n_1+i,n_2-i}\sigma_{n_3}$ will be a sum of Schubert cycles with multiplicity 1, and $\sigma_{d-2,d-2-n_4}$ appears if and only if $d-2-n_4\le n_1+i$. If $d-2-n_4-n_1\le 0$, or equivalently $d-n_1-1\le n_4+1$, then this is true for all of the terms above, and we conclude that 
\begin{equation*}
\int_{\Gr(2,d)}\sigma_{n_1}\sigma_{n_2}\sigma_{n_3}\sigma_{n_4}=d-n_1-1.
\end{equation*}
Otherwise, the number of terms for which $d-2-n_4\le n_1+i$ is $n_4+1$, and 
\begin{equation*}
\int_{\Gr(2,d)}\sigma_{n_1}\sigma_{n_2}\sigma_{n_3}\sigma_{n_4}=n_4+1.
\end{equation*}
This establishes the lemma.
\end{proof}

\begin{lem}\label{integral_against_magic_class}
Let $n_1\ge n_2\ge n_3\ge n_4\ge0$ be integers satisfying $n_1+n_2+n_3+n_4=2d-4$.
\begin{equation*}
\int_{\Gr(2,d+1)}\sigma_{n_1}\sigma_{n_2}\sigma_{n_3}\sigma_{n_4}(8\sigma_{11}-2\sigma_1^2)=
        \begin{cases} 
6 & n_1=n_2=n_3=n_4\\
4 & n_1 = n_2 \neq n_3 = n_4\\
2 & n_1+n_4 = n_2+n_3\text{ and }n_1\neq n_2\\
-2 & n_1 = n_2+n_3+n_4+2\\
0 & \text{otherwise}
   \end{cases}
   \end{equation*}
\end{lem}

\begin{proof}
Without loss of generality, suppose that $n_1\ge n_2\ge n_3\ge n_4$. First, if $n_1>d-1$, then $\sigma_{n_1}=0$, and it is clear that none of the first four conditions on the right hand side can be satisfied. If $n_1=d-1$, then we are in the fourth case on the right hand side, as $n_2+n_3+n_4=(2d-4)-(d-1)=d-3$. In this case, the Pieri rule implies that $\sigma_{d-1}\sigma_{11}=0$ and
\begin{equation*}
\int_{\Gr(2,d+1)}\sigma_{d-1}\sigma_{n_2}\sigma_{n_3}\sigma_{n_4}\sigma_{1}^2=1
\end{equation*}
so again the Lemma holds.

We next dispose of the case $n_1\le 1$: the possibilities are $(n_1,n_2,n_3,n_4)=(0,0,0,0),(1,1,0,0),(1,1,1,1)$, and one easily checks that the Lemma holds here.

Thus, we assume that the $d-2\ge n_1\ge n_2\ge n_3\ge n_4\ge0$ and $n_1\ge 2$. Applying the Pieri rule and Lemma \ref{integral_fourfold_product},
\begin{align*}
&\quad\int_{\Gr(2,d+1)}\sigma_{n_1}\sigma_{n_2}\sigma_{n_3}\sigma_{n_4}(8\sigma_{11}-2\sigma_1^2)\\
&=\int_{\Gr(2,d+1)}\sigma_{n_1}\sigma_{n_2}\sigma_{n_3}\sigma_{n_4}(6\sigma_{11}-2\sigma_{11})\\
&=6\int_{\Gr(2,d)}\sigma_{n_1}\sigma_{n_2}\sigma_{n_3}\sigma_{n_4}-2\int_{\Gr(2,d+1)}(\sigma_{n_1+2}+\sigma_{n_1+1,1}+\sigma_{n_1,2})\sigma_{n_2}\sigma_{n_3}\sigma_{n_4}\\
&=4\int_{\Gr(2,d)}\sigma_{n_1}\sigma_{n_2}\sigma_{n_3}\sigma_{n_4}-2\int_{\Gr(2,d+1)}\sigma_{n_1+2}\sigma_{n_2}\sigma_{n_3}\sigma_{n_4}-2\int_{\Gr(2,d-1)}\sigma_{n_1-2}\sigma_{n_2}\sigma_{n_3}\sigma_{n_4}\\
&=4\min(d-n_1-1,n_4+1)-2\min(d-n_1-2,n_4+1)-2\int_{\Gr(2,d-1)}\sigma_{n_1-2}\sigma_{n_2}\sigma_{n_3}\sigma_{n_4}.
\end{align*}

We now consider the first, second, third, and fifth cases separately: as $n_1\le d-2$, we cannot have $n_1=n_2+n_3+n_4+2$. Suppose first that $n_1=n_2=n_3=n_4=n$, and $d=2n+2$. We then have 
\begin{align*}
&\quad4\min(d-n_1-1,n_4+1)-2\min(d-n_1-2,n_4+1)-2\int_{\Gr(2,d-1)}\sigma_{n_1-2}\sigma_{n_2}\sigma_{n_3}\sigma_{n_4}\\
&=4\min(n,n+1)-2\min(n+1,n+1)-2\int_{\Gr(2,2n+1)}\sigma_n^3\sigma_{n-2}\\
&=4n-2(n+1)-2\min(n,n-1)\\
&=6.
\end{align*}

Next, consider the case $n_1=n_2\neq n_3=n_4$, so that $d=n_1+n_3+2$. We have
\begin{align*}
&\quad4\min(d-n_1-1,n_4+1)-2\min(d-n_1-2,n_4+1)-2\int_{\Gr(2,d-1)}\sigma_{n_1-2}\sigma_{n_2}\sigma_{n_3}\sigma_{n_4}\\
&=4\min(n_3+1,n_4+1)-2\min(n_3,n_4+1)-2\int_{\Gr(2,n_1+n_3+1)}\sigma_{n_1}\sigma_{n_1-2}\sigma_{n_3}^2.
\end{align*}
To evaluate the last term, we consider two sub-cases: if $n_1-2\ge n_3$, then by Lemma \ref{integral_fourfold_product}, we have
\begin{equation*}
\int_{\Gr(2,n_1+n_3+1)}\sigma_{n_1}\sigma_{n_1-2}\sigma_{n_3}^2=\min(n_3,n_3+1)=n_3.
\end{equation*}
On the other hand, if $n_1-2<n_3$, we must have $n_1-n_3=1$, as $n_1>n_3$. Thus, 
\begin{equation*}
\int_{\Gr(2,n_1+n_3+1)}\sigma_{n_1}\sigma_{n_1-2}\sigma_{n_3}^2=\min(n_1,n_1-1)=n_1-1=n_3.
\end{equation*}
Therefore, in both sub-cases, we have 
\begin{equation*}
\int_{\Gr(2,d+1)}\sigma_{n_1}\sigma_{n_2}\sigma_{n_3}\sigma_{n_4}(8\sigma_{11}-2\sigma_1^2)=4(n_3+1)-2n_3-2n_3=4.
\end{equation*}

Next, consider the case $n_1+n_4=n_2+n_3$ and $n_1\neq n_2$. Then, $d-n_1=n_4+2$. Thus,
\begin{align*}
&\quad4\min(d-n_1-1,n_4+1)-2\min(d-n_1-2,n_4+1)-2\int_{\Gr(2,d-1)}\sigma_{n_1-2}\sigma_{n_2}\sigma_{n_3}\sigma_{n_4}\\
&=4(n_4+1)-2n_4-2\int_{\Gr(2,d-1)}\sigma_{n_1-2}\sigma_{n_2}\sigma_{n_3}\sigma_{n_4}.
\end{align*}
If $n_1-n_2\ge 2$, then the last integral is equal to $\min(d-n_1,n_4+1)=(n_4+1)$, and we immediately deduce the lemma. If, on the other hand, $n_1-n_2=1$, we first note that $n_4\le n_1-2$, or else $n_2=n_3=n_4$, an impossibility. Then, Lemma \ref{integral_fourfold_product} implies that the last term is again equal to $\min(d-n_2-1,n_4+1)=\min(d-n_1,n_4+1)=n_4+1$, so we are done in this case.\\

Finally, suppose $n_1+n_4\neq n_2+n_3$. In particular, we have either $d-n_1-2\ge n_4+1$ or $d-n_1\le n_4+1$. First, assume that $n_1-n_2\ge 2$. Then,
\begin{equation*}
\int_{\Gr(2,d-1)}\sigma_{n_1-2}\sigma_{n_2}\sigma_{n_3}\sigma_{n_4}=\min(d-n_1,n_4+1).
\end{equation*}
Thus, the expression
\begin{equation}\label{three_terms}
4\min(d-n_1-1,n_4+1)-2\min(d-n_1-2,n_4+1)-2\int_{\Gr(2,d-1)}\sigma_{n_1-2}\sigma_{n_2}\sigma_{n_3}\sigma_{n_4}
\end{equation}
is equal to either
\begin{equation*}
4(n_4+1)-2(n_4+1)-2(n_4+1)=0
\end{equation*}
or
\begin{equation*}
4(d-n_1-1)-2(d-n_1-2)-2(d-n_1)=0,
\end{equation*}
so we have the lemma if $n_1-n_2\ge 2$.

Suppose instead that $n_1-n_2=0$ or $n_1-n_2=1$. Then, we can check as before that $n_1-2\ge n_4$. Furthermore, we claim that we must have $d-n_1-2\ge n_4+1$. If not, then we have instead $d-n_1\le n_4+1$, so $n_2+n_4\ge (n_1-1)+n_4\ge d-2=\frac{1}{2}(n_1+n_2+n_3+n_4)$, which is impossible unless $n_1=n_2,n_3=n_4$. From here, one easily evaluates the expression (\ref{three_terms}) as in the previous cases, so we are done.
\end{proof}

\begin{prop}\label{genus1_laurent}
$N_{d_1,d_2,d_3,d_4}$ is the constant term of the Laurent polynomial $P_{d_1-1}P_{d_2-1}P_{d_3-1}P_{d_4-1}$, where
\begin{equation*}
P_r=rq^n+(r-2)q^{r-2}+\cdots+(-r+2)q^{-r+2}+(-r)q^{-r}.
\end{equation*}
\end{prop}

\begin{proof}
First, observe that, by the Pieri rule,
\begin{equation*}
\tau_{d_i-2}:=\sum_{a_i+b_i=d_i-2}\sigma_{a_i}\sigma_{b_i}=\sum_{a'_i+b'_i=d_i-2}(a'_i-b'_i+1)\sigma_{a'_i,b'_i}.
\end{equation*}

By Lemma \ref{integral_against_magic_class}, the positive contributions to
\begin{equation*}
\int_{\Gr(2,d+1)}\tau_{d_1-2}\tau_{d_2-2}\tau_{d_3-2}\tau_{d_4-2}(8\sigma_{11}-2\sigma_1^2)
\end{equation*}
correspond to terms $\sigma_{a'_1,b'_1}\sigma_{a'_2,b'_2}\sigma_{a'_3,b'_3}\sigma_{a'_4,b'_4}$ with
\begin{equation*}
(a'_i-b'_i)+(a'_j-b'_j)=(a'_k-b'_k)+(a'_\ell-b'_\ell)
\end{equation*}
for some permutation $(i,j,k,\ell)$ of $(1,2,3,4)$, Moreover, if we fix $i=1$, the contribution to the integral is
\begin{equation*}
2m\prod_{i=1}^4(a'_i-b'_i+1)
\end{equation*}
where $m$ is the number of such permutations. Similarly, the negative contributions to the integral correspond to terms where
\begin{equation*}
(a'_i-b'_i+1)=(a'_j-b'_j+1)+(a'_k-b'_k+1)+(a'_\ell-b'_\ell+1),
\end{equation*} 
and the contribution to the integral is 
\begin{equation*}
-2\prod_{i=1}^4(a'_i-b'_i+1)
\end{equation*}

We match these contributions exactly with the contributions to the constant term in $P_{d_1-1}P_{d_2-1}P_{d_3-1}P_{d_4-1}$: the positive contributions come from terms
\begin{equation*}
\prod_{i=1}^4m_iq^{m_i}
\end{equation*}
with exactly two of the $m_i$ positive, and the negative contributions come from terms in which one or three of the $m_i$ are positive.
\end{proof}

One can easily deduce the following, which is also a consequence of Proposition \ref{first_point_unweighted}.

\begin{cor}
Suppose that $d_1=d$. Then,
\begin{equation*}
N_{d_1,d_2,d_3,d_4}=2(d+1)(d_2-1)(d_3-1)(d_4-1).
\end{equation*}
\end{cor}

In particular, we recover \cite[Theorem 2.1(f)]{harris}, as well as the fact that the number of degree $d$ covers $f:E\to\bP^1$ totally ramified at a fixed point of $E$ and one other point is equal to $\#E[d^2]-1=d^2-1$.


%
%

\vskip 3pt

\subsection{Explicit formula, and proof of Theorem \ref{genus1_duality}}\label{genus1_unweighted_explicit}

Using the Laurent polynomial formula of the previous section, we now complete the proofs of Theorems \ref{genus1_unweighted_count} and \ref{genus1_duality}.

Let $m,n$ be integers with $m\ge n$. We first compute $P_{m}P_{n}$. When $0\le k\le n$, the coefficients of $q^{m+n-2k}$ and $q^{-m-n+2k}$ in $P_{m}P_{n}$ are
\begin{align*}
\sum_{r+s=k}(m-2r)(n-2s)&=(k+1)mn-k(k+1)(m+n)+4\sum_{i=0}^{k}i(k-i)\\
&=(k+1)mn-k(k+1)(m+n)+4\left(\frac{k^2(k+1)}{2}-\frac{k(k+1)(2k+1)}{6}\right)\\
&=(k+1)\left[mn-k(m+n)+\frac{2}{3}k(k-1)\right]
\end{align*}

When $0\le k\le m-n$, the coefficients of $q^{m-n-2k}$ are
\begin{align*}
&(m-2k)(-n)+(m-2k-2)(-n+2)+\cdots+(m-2k-2n+2)(n-2)+(m-2k-2n)(n)\\
=&\frac{1}{2}\left[(2n)(-n)+(2n-2)(-n+2)+\cdots+(-2n+2)(n-2)+(-2n)(n)\right],
\end{align*}
where we have paired summands from the outside inward. In particular, the value of this coefficient does not depend on $k$, so we may take $k=0$, in which case we have already computed the $q^{m-n}$ coefficient to be
\begin{align*}
&(n+1)\left[mn-n(m+n)+\frac{2}{3}n(n-1)\right]\\
=&(n+1)\left[-\frac{1}{3}n^2-\frac{2}{3}n\right]\\
=&-\frac{1}{3}n(n+1)(n+2).
\end{align*}
Also, the coefficients of $q^{r}$ and $q^{-r}$ are equal for all $r$.

We now evaluate the constant term of $(P_{d_1-1}P_{d_2-1})\cdot (P_{d_3-1}P_{d_4-1})$ by matching coefficients in the two factors. Without loss of generality, assume that $d\ge d_1\ge d_2\ge d_3\ge d_4\ge 0$. We consider the coefficients $q^\ell$ in the first term and $q^{-\ell}$ in the second with $-d_3-d_4+2\le \ell\le d_{3}+d_{4}-2$. 

First, suppose that $d_1-d_2\ge d_3-d_4$. Then, there are five intervals over which we vary $\ell$: $[-d_3-d_4+2,d_2-d_1]$, $(d_2-d_1,d_4-d_3]$, $(d_4-d_3,d_3-d_4)$, $[d_3-d_4,d_1-d_2)$, and $[d_1-d_2,d_3+d_4-2]$.

The contribution from the interval $(d_4-d_3,d_3-d_4)$ to $N_{d_1,d_2,d_3,d_4}$ is
\begin{equation*}
\left(-\frac{1}{3}d_2(d_2+1)(d_2-1)\right)\cdot\left(-\frac{1}{3}d_4(d_4+1)(d_4-1)\right)\cdot(d_3-d_4-1)
\end{equation*}

The contribution from $(d_2-d_1,d_4-d_3]$ and $[d_3-d_4,d_1-d_2)$ is
\begin{dmath*}
2\left(-\frac{1}{3}d_2(d_2+1)(d_2-1)\right)\cdot\sum_{j=1}^{\frac{(d_1-d_2)-(d_3-d_4)}{2}}(d_4-j+1)\left[(d_3-1)(d_4-1)-(d_4-j)(d_3+d_4-2)+\frac{2}{3}(d_4-j)(d_4-j-1)\right]
\end{dmath*}

The contribution from $[-d_3-d_4+2,d_2-d_1]$ and $[d_1-d_2,d_3+d_4-2]$ is
\begin{dmath*}
2\sum_{k=0}^{\frac{d_2+d_3+d_4-d_1}{2}-1}\left\{(k+1)\left[(d_3-1)(d_4-1)-k(d_3+d_4-2)+\frac{2}{3}k(k-1)\right]\cdot(k'+1)\left[(d_1-1)(d_2-1)-k'(d_1+d_2-2)+\frac{2}{3}k'(k'-1)\right]\right\}
\end{dmath*}
where $k'=k+\frac{(d_1+d_2)-(d_3+d_4)}{2}$.

Summing these contributions yields a formula for $N_{d_1,d_2,d_3,d_4}$ in the case $d_1-d_2\ge d_3-d_4$: we may expand each summand and apply standard formulas for sums of $m$-th powers of integers for $m\le 6$. This is implemented with the help of SAGE, and yields:
\begin{dmath}\label{N_explicit_case1}
N_{d_1,d_2,d_3,d_4}=-\frac{1}{3360}d_1^7 + \frac{1}{240}d_1^5d_2^2 - \frac{1}{96}d_1^4d_2^3 + \frac{1}{96}d_1^3d_2^4 - \frac{1}{240}d_1^2d_2^5 + \frac{1}{3360}d_2^7 + \frac{1}{240}d_1^5d_3^2 - \frac{1}{48}d_1^3d_2^2d_3^2+ \frac{1}{48}d_1^2d_2^3d_3^2 - \frac{1}{240}d_2^5d_3^2 - \frac{1}{96}d_1^4d_3^3 + \frac{1}{48}d_1^2d_2^2d_3^3 - \frac{1}{96}d_2^4d_3^3 + \frac{1}{96}d_1^3d_3^4 - \frac{1}{96}d_2^3d_3^4- \frac{1}{240}d_1^2d_3^5 - \frac{1}{240}d_2^2d_3^5 + \frac{1}{3360}d_3^7 + \frac{1}{240}d_1^5d_4^2 - \frac{1}{48}d_1^3d_2^2d_4^2+ \frac{1}{48}d_1^2d_2^3d_4^2- \frac{1}{240}d_2^5d_4^2 - \frac{1}{48}d_1^3d_3^2d_4^2 + \frac{1}{48}d_2^3d_3^2d_4^2 + \frac{1}{48}d_1^2d_3^3d_4^2 + \frac{1}{48}d_2^2d_3^3d_4^2 - \frac{1}{240}d_3^5d_4^2 - \frac{1}{96}d_1^4d_4^3 + \frac{1}{48}d_1^2d_2^2d_4^3- \frac{1}{96}d_2^4d_4^3 + \frac{1}{48}d_1^2d_3^2d_4^3 + \frac{1}{48}d_2^2d_3^2d_4^3 - \frac{1}{96}d_3^4d_4^3 + \frac{1}{96}d_1^3d_4^4 - \frac{1}{96}d_2^3d_4^4 - \frac{1}{96}d_3^3d_4^4 - \frac{1}{240}d_1^2d_4^5 - \frac{1}{240}d_2^2d_4^5 - \frac{1}{240}d_3^2d_4^5+ \frac{1}{3360}d_4^7 - \frac{1}{480}d_1^5 + \frac{1}{96}d_1^4d_2 - \frac{1}{48}d_1^3d_2^2 + \frac{1}{48}d_1^2d_2^3 - \frac{1}{96}d_1d_2^4 + \frac{1}{480}d_2^5 + \frac{1}{96}d_1^4d_3 - \frac{1}{48}d_1^2d_2^2d_3 + \frac{1}{96}d_2^4d_3 - \frac{1}{48}d_1^3d_3^2- \frac{1}{48}d_1^2d_2d_3^2 + \frac{1}{48}d_1d_2^2d_3^2 + \frac{1}{48}d_2^3d_3^2 + \frac{1}{48}d_1^2d_3^3 + \frac{1}{48}d_2^2d_3^3 - \frac{1}{96}d_1d_3^4 + \frac{1}{96}d_2d_3^4 + \frac{1}{480}d_3^5 + \frac{1}{96}d_1^4d_4 - \frac{1}{48}d_1^2d_2^2d_4+ \frac{1}{96}d_2^4d_4 - \frac{1}{48}d_1^2d_3^2d_4 - \frac{1}{48}d_2^2d_3^2d_4 + \frac{1}{96}d_3^4d_4 - \frac{1}{48}d_1^3d_4^2 - \frac{1}{48}d_1^2d_2d_4^2 + \frac{1}{48}d_1d_2^2d_4^2 + \frac{1}{48}d_2^3d_4^2 - \frac{1}{48}d_1^2d_3d_4^2 - \frac{1}{48}d_2^2d_3d_4^2+ \frac{1}{48}d_1d_3^2d_4^2 - \frac{1}{48}d_2d_3^2d_4^2 + \frac{1}{48}d_3^3d_4^2 + \frac{1}{48}d_1^2d_4^3 + \frac{1}{48}d_2^2d_4^3 + \frac{1}{48}d_3^2d_4^3 - \frac{1}{96}d_1d_4^4 + \frac{1}{96}d_2d_4^4 + \frac{1}{96}d_3d_4^4 + \frac{1}{480}d_4^5 + \frac{1}{60}d_1^3-\frac{1}{60}d_1^2d_2 + \frac{1}{60}d_1d_2^2 - \frac{1}{60}d_2^3 - \frac{1}{60}d_1^2d_3 - \frac{1}{60}d_2^2d_3 + \frac{1}{60}d_1d_3^2 - \frac{1}{60}d_2d_3^2 - \frac{1}{60}d_3^3 - \frac{1}{60}d_1^2d_4 - \frac{1}{60}d_2^2d_4- \frac{1}{60}d_3^2d_4 + \frac{1}{60}d_1d_4^2-\frac{1}{60}d_2d_4^2 - \frac{1}{60}d_3d_4^2 - \frac{1}{60}d_4^3 - \frac{1}{70}d_1 + \frac{1}{70}d_2 + \frac{1}{70}d_3 + \frac{1}{70}d_4
\end{dmath}

Now, consider the case  $d_1-d_2\le d_3-d_4$. Similarly to the first case, the contribution from the interval $(d_2-d_1,d_1-d_2)$ to $N_{d_1,d_2,d_3,d_4}$ is
\begin{dmath*}
\left(-\frac{1}{3}d_2(d_2+1)(d_2-1)\right)\cdot\left(-\frac{1}{3}d_4(d_4+1)(d_4-1)\right)\cdot(d_1-d_2-1)
\end{dmath*}

The contribution from $(d_4-d_3,d_2-d_1]$ and $[d_1-d_2,d_3-d_4)$ is
\begin{dmath*}
2\left(-\frac{1}{3}d_4(d_4+1)(d_4-1)\right)\cdot\sum_{j=1}^{\frac{(d_3-d_4)-(d_1-d_2)}{2}}(d_2-j+1)\left[(d_1-1)(d_2-1)-(d_2-j)(d_1+d_2-2)+\frac{2}{3}(d_2-j)(d_2-j-1)\right]
\end{dmath*}

Finally, the contribution from $[-d_3-d_4+2,d_4-d_3]$ and $[d_3-d_4,d_3+d_4-2]$ is
\begin{dmath*}
\sum_{k=0}^{d_4-1}(k+1)\left[(d_3-1)(d_4-1)-k(d_3+d_4-2)+\frac{2}{3}k(k-1)\right](k'+1)\left[(d_1-1)(d_2-1)-k'(d_1+d_2-2)+\frac{2}{3}k'(k'-1)\right],
\end{dmath*}
where $k'=k+\frac{(d_1+d_2)-(d_3+d_4)}{2}$.

Summing as in the previous case, we get that $N_{d_1,d_2,d_3,d_4}$ is equal to:

\begin{dmath}\label{N_explicit_case2}
N_{d_1,d_2,d_3,d_4}=-\frac{1}{48}d_1^4d_4^3 + \frac{1}{24}d_1^2d_2^2d_4^3 - \frac{1}{48}d_2^4d_4^3 + \frac{1}{24}d_1^2d_3^2d_4^3 + \frac{1}{24}d_2^2d_3^2d_4^3 - \frac{1}{48}d_3^4d_4^3 - \frac{1}{120}d_1^2d_4^5 - \frac{1}{120}d_2^2d_4^5 - \frac{1}{120}d_3^2d_4^5+ \frac{1}{1680}d_4^7 + \frac{1}{48}d_1^4d_4 - \frac{1}{24}d_1^2d_2^2d_4 + \frac{1}{48}d_2^4d_4 - \frac{1}{24}d_1^2d_3^2d_4 - \frac{1}{24}d_2^2d_3^2d_4 + \frac{1}{48}d_3^4d_4 + \frac{1}{24}d_1^2d_4^3 + \frac{1}{24}d_2^2d_4^3 + \frac{1}{24}d_3^2d_4^3+ \frac{1}{240}d_4^5 - \frac{1}{30}d_1^2d_4 - \frac{1}{30}d_2^2d_4 - \frac{1}{30}d_3^2d_4 - \frac{1}{30}d_4^3 + \frac{1}{35}d_4
\end{dmath}

\begin{proof}[Proof of Theorem \ref{genus1_unweighted_count}]
Combine the above with Propositions \ref{N_schubert_formula} and \ref{genus1_laurent}.
\end{proof}

\begin{proof}[Proof of Theorem \ref{genus1_duality}]
One checks by direct computation (carried out in SAGE) that the right hand sides of (\ref{N_explicit_case1}) and (\ref{N_explicit_case2}) are sent to each other under the involution $d_i\mapsto \frac{1}{2}(d_1+d_2+d_3+d_4)-d_i$.
\end{proof}

\vskip 4pt

\section{The general case via limit linear series}\label{degeneration_section}

In this section, we use the theory of limit linear series to give a more precise version of Theorem \ref{degeneration}, which yields explicit answers to Question \ref{main_question} for any given values of $g,d,d_i$. A similar degeneration technique is used in \cite{logan,osserman,fmnp}.

\vskip 3pt

\subsection{The degeneration formula}

We adopt the notation of Question \ref{main_question} and assume that (\ref{total_ramification_conditions}) holds. In addition, we take $m=3g$, following Proposition \ref{number_of_moving_points} and the ensuing discussion. Then, let $N^g_{d_1,\ldots,d_{n+3g}}$ be the answer to Question \ref{main_question}, counting covers $f:C\to\bP^1$ with ramification index $d_i$ at fixed points $p_1,\ldots,p_n$ and moving points $p_{n+1},\ldots,p_{n+3m}$.

\begin{defn}\label{X0_def}
Fix general elliptic curves $(E_j,q_j)$, $j=1,2,\ldots,g$, and fix a general $(n+g)$-pointed rational curve $(\bP^1,p_1,\ldots,p_n,r_1,\ldots,r_g)$. Then, let $(X_0,p_1,\ldots,p_n)$ be the nodal curve obtained by attaching the $E_j$ to $\bP^1$, gluing the point $r_j$ to $q_j$ for $j=1,2,\ldots,g$.
\end{defn}

\begin{lem}\label{LLS_comb_structure}
Consider the moduli space $\cG_{X_0}$ of tuples $(V_0,p_{n+1},\ldots,p_{n+3g})$, where $V_0$ is a limit linear series of degree $d$ on $X_0$, and $p_1,\ldots,p_{3g}\in X$ are pairwise distinct smooth points of $X_0$ such that  $V_0$ has vanishing at least $(0,d_i)$ at $p_i$. Then, we have:
\begin{enumerate}
\item[(a)] Given any $[(V_0,p_{n+1},\ldots,p_{n+3g})]\in\cG_{X_0}$, $V_0$ is \textit{refined} (in the sense of \cite{eh_lls}), and exactly three of the moving points $p_i$, $i=n+1,\ldots,n+3g$ lie on each $E_j$. Moreover, the vanishing sequence of $V_0$ at $p_i$ is exactly $(0,d_i)$ for all $i=1,2,\ldots,n+3g$. In particular, none of $p_{n+1},\ldots,p_{n+3g}$ lie on $\bP^1$.
\item[(b)] $\cG_{X_0}$ is reduced of dimension 0.
\item[(c)] Any $[(V_0,p_{n+1},\ldots,p_{n+3g})]\in\cG_{X_0}$ smooths to a linear series on the general fiber of the versal deformation of $(X_0,p_{1},\ldots,p_{n+3g})$, preserving the ramification conditions at the $p_i$.
\end{enumerate}
\end{lem}

\begin{proof}
By condition (\ref{total_ramification_conditions}), we always have $\rho(V_0,\{p_1,\ldots,p_{n+3g}\})=-3g$. By sub-additivity of the Brill-Noether number (\ref{bn_number_additive}) and Proposition \ref{expected_dim}, we have that $\rho(V_0,\{p_i\})_{E_j}=-3$ for all $j$, and $\rho(V_0,\{p_{i}\})_{\bP^1}=0$. Thus, the Brill-Noether number is in fact additive, so $V_0$ is a refined limit linear series. Moreover, it follows that we need three moving points on each $E_j$, and that $V_0$ cannot have higher-than-expected ramification at any of the $p_i$; this establishes (a). 

Part (b) follows from the same statements for the moduli of linear series on the individual components; on the rational spine, this is Theorem \ref{genus0_count}, and on the elliptic components, this is a consequence of the transversality argument given in Lemma \ref{identify_pencil_weights}. Finally, part (c) follows immediately from \cite[Corollary 3.7]{eh_lls}, as $V_0$ is refined, and dimensionally proper with respect to the $p_i$.
\end{proof}

\begin{lem}\label{LLS_specialization}
Let $R$ be a discrete valuation ring, and let $B=\Spec R$. Let $\pi:X\to B, \sigma_i: B\to X$, $i=1,2,\ldots,n$ be a 1-family of $n$-pointed, genus $g$ curves with special fiber isomorphic to $(X_0,p_1,\ldots,p_n)$ and smooth total space $X$. Let $p'_i$ denote the restriction of $\sigma_i$ to the geometric generic fiber $X_{\overline{\eta}}$ for $i=1,2,\ldots,n$. Suppose that $(V',p'_{n+1},\ldots,p'_{n+3g})$ is a tuple where $V'$ is a linear series on $X_{\overline{\eta}}$ and $p'_1,\ldots,p'_{n+3g}\in X_{\overline{\eta}}$ are pairwise distinct points such that $V'$ has ramification sequence $(0,d_i)$ at $p'_i$. Then, $(V',p'_{n+1},\ldots,p'_{n+3g})$ specializes to a tuple $(V_0,p_{n+1},\ldots,p_{n+3g})$ as in Lemma \ref{LLS_comb_structure}.
\end{lem}

\begin{proof}
The content of the lemma is that the $p'_{i}$, $i=n+1,\ldots,n+3g$ specialize to distinct smooth points of the special fiber. Suppose that this is not the case: then, after a combination of blow-ups and base-changes, the $p'_i$ specialize to distinct smooth points on a compact-type curve $Y_0$ with a non-trivial map $c:Y_0\to X_0$ contracting rational tails and bridges. Moreover, $Y_0$ is equipped with a limit linear series $W_0$ with ramification conditions as above at distinct smooth points $p'_i$. As before, we have $\rho(W_0,\{p_i\})=-3g$. Let $E'_j$ denote the unique component of $Y_0$ mapping to $E_j\subset X_0$, and $q'_j\in E'_j$ denote the unique point of $E'_j$ such that $c(q'_j)=q_j$.

We claim that if such a $W_0$ exists, then in fact $Y_0=X_0$. First, note that $\rho(W_0,\{p_i\})_{R}\ge0$ for any rational component $R\subset Y_0$, by Proposition \ref{number_of_moving_points}. Also, the elliptic components of $Y_0$ are general, so by Propositions \ref{expected_dim} and \ref{number_of_moving_points}, the $\rho(W_0,\{p_i\})_{E'_j}\ge-3$ for all $j$. Thus, by sub-additivity of the Brill-Noether number  (\ref{bn_number_additive}), equality must hold everywhere, and moreover $W_0$ is a refined limit.

For each $j$, let $\alpha_j$ be the number of moving points on $E'_j$, and let $\beta_j$ be the number of trees of rational curves attached to $E'_j$ away from $q'_j$. Then, $\rho(W_0,\{p_i\})_{E'_j}=-(\alpha_j+\beta_j)$. Thus,
\begin{equation*}
3g=-\rho(W_0)=\sum_{j}(\alpha_j+\beta_j).
\end{equation*}
On the other hand, each such tree of rational curves attached to an $E'_j$ away from $q'_j$ contains at least two of the $p'_i$, so we have 
\begin{equation*}
\sum_{j}(\alpha_j+2\beta_j)\le 3g.
\end{equation*}
Therefore, $\beta_j=0$ for all $j$, from which it follows that $c$ is an isomorphism. This completes the proof.
\end{proof}

\begin{prop}\label{degeneration_formula_precise}
The answer $N^g_{d_1,\ldots,d_{n+3g}}$ to Question \ref{main_question} is computed in the following way. Consider all distributions 
\begin{equation*}
S=(\{p'_1,p'_2,p'_3\},\{p'_4,p'_5,p'_6\},\ldots,\{p'_{3g-2},p'_{3g-1},p'_{3g}\})
\end{equation*} of the points $p_{n+1},\ldots,p_{n+3g}$ onto the $E_j$ such that each elliptic component contains exactly three of the $p_i$. For each $E_j$, containing the points $r_{3j-2},r_{3j-1},r_{3j}$, consider all possible vanishing sequences $(a_j,b_j)$ such that
\begin{equation*}
(a_j+b_j)+(d_{3j-2}+d_{3j-1}+d_{3j})=2d+4.
\end{equation*}
Then, take the product
\begin{equation*}
\left[\int_{\Gr(2,d+1)}\left(\prod_{j=1}^{g}\sigma_{d-a_j-1,d-b_j}\cdot\prod_{i=1}^{n}\sigma_{d_{i}-1}\right)\right]\cdot\prod_{j=1}^{g}N_{b_j-a_j,d_{3j-2},d_{3j-1},d_{3j}}.
\end{equation*}
Finally, sum the resulting products over all choices of $S,(a_j,b_j)$.
\end{prop}

\begin{proof}
By Lemmas \ref{LLS_comb_structure} and \ref{LLS_specialization}, $N^g_{d_1,\ldots,d_{n+3g}}$ is equal to the number of $(V_0,p_{n+1},\ldots,p_{n+3g})$ as described in Lemma \ref{LLS_comb_structure}(a). To enumerate such limit linear series, we consider all possible $S$ as above, then all possible combinations of vanishing sequences $(a_j,b_j)$ at the nodes $q_j\in E_j$. Then, as $V_0$ is a refined series, the vanishing sequence at $r_j\in \bP^1$ must be $(d-b_j,d-a_j)$. After twisting away base-points at the $q_j$, the terms in the product then count the number of linear series on the components of $X_0$, by Theorems \ref{genus0_count} and \ref{genus1_unweighted_count}.
\end{proof}

\begin{proof}[Proof of Theorem \ref{degeneration}]
Immediate from Proposition \ref{degeneration_formula_precise}.
\end{proof}

Simplifying the degeneration formula of Proposition \ref{degeneration_formula_precise} seems to be a difficult combinatorial problem. It seems natural to guess that in higher genus, weighted counts of pencils are better behaved than unweighted counts of branched covers. More precisely, we consider the following variant of Question \ref{main_question}: 

\begin{question}\label{main_question_weighted}
Let $(C,p_1,\ldots,p_n)$ be a general pointed curve of genus $g$, where $2g-2+n>0$. Let $d,d_1,d_2,\ldots,d_{n+m}$ be integers such that $2\le d_i\le 2d-g-1$ for all $i$. Furthermore, assume that (\ref{total_ramification_conditions}) holds. How many $(m+1)$-tuples $(p_{n+1},\ldots,p_{n+m},V)$ are there, where $p_i\in C$ are pairwise distinct points, $V$ is a degree $d$ pencil on $C$ with total vanishing at least $d_i$ at each $p_i$, and each $(p_{n+1},\ldots,p_{n+m},V)$ is counted with multiplicity 
\begin{equation*}
C^{d_1,\ldots,d_{n+m}}_{k_1,\ldots,k_{n+m}}=\prod_{i=1}^{n+m}c_{d_i-k_i-1,k_i},
\end{equation*}
where $k_i$ is the order of base-point of $V$ at $p_i$, and the $c_{a,b}$ are as in Lemma \ref{LR_coeffs}?
\end{question}

We remark here that in this setting, one gets a degeneration formula for the weighted number of pencils $\wt{N}^g_{d_1,\ldots,d_{n+3g}}$ on $C$ by replacing $N_{b_j-a_j,d_{3j-2}-a_j,d_{3j-1}-a_j,d_{3j}-a_j}$ with $\wt{N}^{\circ}_{b_j-a_j,d_{3j-2}-a_j,d_{3j-1}-a_j,d_{3j}-a_j}$ in Proposition \ref{degeneration_formula_precise}, see Proposition \ref{first_point_unweighted}.

We make one final observation, that in the weighted setting, it suffices to consider the case $n=1$, that is, the case in which there is only one fixed ramification condition.

\begin{prop}\label{consolidate_fixed_points}
The answer to Question \ref{main_question_weighted} remains the same upon replacing $p_1,\ldots,p_n$ by a single general point $p'_1\in C$, at which we impose the condition of total vanishing $d'_1=(d_1+\cdots+d_n)-n+1$.
\end{prop}

\begin{proof}
We degenerate $C$ to the nodal curve $C_0\cong C\cup\bP^1$ so that that the points $p_1,\ldots,p_n$ specialize to general points on $\bP^1$, and count limit linear series on the pointed curve $(C_0,p'_1)$, where $p'_1=C\cap\bP^1$. The details are left to the reader.
\end{proof}

\end{document}